\documentclass[11pt,a4paper]{amsart}
\usepackage{amsfonts}
\usepackage{amsmath}
\usepackage{amstext}
\usepackage{amsbsy}
\usepackage{amsthm}
\usepackage{amsopn}

\usepackage{comment}
\usepackage{amssymb}
\usepackage{amsmath}
\usepackage{latexsym}
\usepackage{mathrsfs}
\usepackage{amsthm}
\usepackage{graphicx}
\usepackage{caption}
\usepackage{enumerate}
\usepackage{hyperref}
\usepackage{color}
\usepackage{enumerate}
\usepackage{hyperref}
\allowdisplaybreaks[1]

\numberwithin{equation}{section}
\newtheorem{thm}[equation]{Theorem}

\newtheorem{rem}[equation]{Remark}
\newtheorem{lemma}[equation]{Lemma}
\newtheorem{corollary}[equation]{Corollary}
\newtheorem{example}[equation]{Example}

\newcommand{\rea}{\mathbb{R}}

\begin{document}

\title[Eigenvalues of elliptic operators with density]
{Eigenvalues of elliptic operators with density}

\author{Bruno Colbois, Luigi Provenzano}

\address{Bruno Colbois, Laboratoire de Math\'ematiques, Universit\'e de Neuch\^atel, 13 Rue E. Argand, 2007 Neuch\^atel, Switzerland. {\it E-mail}: {\rm bruno.colbois@unine.ch}}

\address{Luigi Provenzano, EPFL, SB Institute of Mathematics, Station 8, CH-1015 Lausanne, Switzerland. {\it E-mail}: {\rm luigi.provenzano@epfl.ch}}



 
 
 \thanks{The second author is member of the Gruppo Nazionale per l'Analisi Matematica, la Probabilit\`a e le loro Applicazioni (GNAMPA) of the Istituto Nazionale di Alta Matematica (INdAM)}

\begin{abstract}
We consider eigenvalue problems for elliptic operators of arbitrary order $2m$ subject to Neumann boundary conditions on bounded domains of the Euclidean $N$-dimensional space. We study the dependence of the eigenvalues upon variations of mass density and in particular we discuss the existence and characterization of upper and lower bounds under both the condition that the total mass is fixed and the condition that the $L^{\frac{N}{2m}}$-norm of the density is fixed. We highlight that the interplay between the order of the operator and the space dimension plays a crucial role in the existence of eigenvalue bounds.
\end{abstract}
\maketitle


\section{Introduction}

We consider the eigenvalue problem
\begin{equation}\label{DENSITYPROBLEM}
(-\Delta)^m u =\mu\rho u
\end{equation}
on a connected bounded open subset $\Omega$ of $\mathbb R^N$, where $\rho$ is a positive function bounded away from zero and infinity and where we impose Neumann boundary conditions on $u$. Under suitable regularity assumptions on the boundary of $\Omega$ (e.g., if $\Omega$ has a Lipschitz boundary) it is standard to prove that problem \eqref{DENSITYPROBLEM} admits an increasing sequence of non-negative eigenvalues of finite multiplicity
$$
0=\mu_1[\rho]=\cdots=\mu_{d_{N,m}}[\rho]<\mu_{d_{N,m}+1}[\rho]\leq\cdots\leq\mu_j[\rho]\cdots\nearrow+\infty,
$$
where $d_{N,m}$ denotes the dimension of the space of polynomials of degree at most $m-1$ in $\mathbb R^N$. 

In this paper we will prove a few results on the dependence of the eigenvalues $\mu_j[\rho]$ upon variation of $\rho$. In particular we will consider the problem of finding upper bounds for $\mu_j[\rho]$ among all positive and bounded densities $\rho$ satisfying suitable constraints. We shall also consider the issue of lower bounds, which also presents some interesting features.

Keeping in mind important problems for the Laplace and the biharmonic operators in linear elasticity (see e.g., \cite{cohil}) we shall think of the weight $\rho$ as a mass density of the body $\Omega$ and we shall refer to the quantity $M=\int_{\Omega}\rho dx$ as to the total mass of $\Omega$. In fact, when $N=2$ the eigenvalues $\mu_j[\rho]$ describe the vibrations of a non-homogeneous membrane with free edge when $m=1$ (see e.g., \cite[\S\,9]{Henrot}) and of a non-homogeneous plate with free edge when $m=2$ (see \cite{buosoprovenzano,chasman}).

Relevant questions on the dependence of the eigenvalues $\mu_j[\rho]$ upon $\rho$ are whether it is possible to minimize or maximize the eigenvalues under the assumption that the total mass is fixed, or whether it is possible to have uniform upper or lower bounds for the eigenvalues (i.e., bounds which depend only on the total mass, the dimension and the eigenvalue index) under the same constraint, and which have the correct behavior in $j\in\mathbb N$ as described by the Weyl's asymptotic law. 

Most of the existing literature treats the case of the Laplace operator with Dirichlet boundary conditions. In particular, we recall the famous result of Krein \cite{krein} on the eigenvalues of the Dirichlet Laplacian in one dimension (fixed string) which completely answers the questions raised above. In fact he finds sharp upper and lower bounds which depend only on $j,M,H,l$ for all the eigenvalues of the Laplacian on the string $]0,l[$ upon densities $0\leq\rho\leq H$ for which $M=\int_0^l\rho dx$ is fixed (see Remark \ref{remKR}). We refer also to the extensive work of Banks and collaborators for generalizations and extensions of Krein's results (see \cite{banks5,banks4,banks3,banks2,banks1} and the references therein). We mention also \cite[\S\,5]{egorov} which contains a detailed analysis of the eigenvalues of Sturm-Liouville problems with Dirichlet conditions with density (and also other types of weight). In particular, in \cite[\S\,5]{egorov}, the authors provide estimates (upper and lower bounds) under various type of linear and non-linear constraints on the weights. Existence of  minimizers and maximizers under mass constraint in higher dimensions for the Dirichlet Laplacian has been investigated in \cite{chanillo,cox0,cox1,cox2,friedland1}, where the authors impose the additional constraint that admissible densities are uniformly bounded from below and above by some fixed constants. We refer to \cite[\S\,9]{Henrot} and to the references therein for further discussions on eigenvalue problems for inhomogeneous strings and membranes with fixed edges.

As for Neumann boundary conditions, much less is known. Very recently the problem of finding uniform upper bounds for the Neumann eigenvalues of the Laplacian with weights has been solved (for $N\geq 2$) by Colbois and El Soufi \cite{colbois_elsoufi} in the more general context of Riemannian manifolds, by exploiting a general result of decomposition of a metric measure space by annuli (see \cite{gny}, see also \cite{korevaar}). The authors have not considered the case $N=1$ and, in fact, as we shall see in the present paper, upper bounds with mass constraint do not exists in dimension one. 

There are very few results for what concerns higher order operators. We recall \cite{banksrods,schwarzrods}, where the authors consider the case of the biharmonic operator in one dimension with intermediate boundary conditions (hinged rod) and \cite{beesackrods}, where the author considers the case of the biharmonic operator with Dirichlet conditions in dimension one (clamped rod) and two (clamped plate). We also refer to \cite[\S\,7.9]{egorov} where it is possible to find some estimates for the eigenvalues of elliptic operators of order $2m$ with density subject to Dirichlet boundary conditions. We refer again to \cite[\S\,11]{Henrot} for a more detailed discussion on eigenvalue problems for inhomogeneous rods and plates with hinged and clamped edges. Up to our knowledge, there are no results in the literature on the existence and characterization of upper and lower bounds with respect to mass densities for higher order operators subject to Neumann boundary conditions (already for the biharmonic operator or the Laplacian in dimension one).

Finally, we refer to \cite{laproeurasian} where the authors prove continuity and differentiability results for the dependence of the eigenvalues of a quite wide class of higher order elliptic operators and homogeneous boundary conditions upon variation of the mass density and in most of the cases (except, again, that of Neumann boundary conditions), they establish a maximum principle for extremum problems related to mass density perturbations which preserve the total mass. We remark that in \cite{laproeurasian} partial results are obtained in the case of Neumann boundary conditions only for the Laplace operator.

In this paper we shall primarily address the issue of finding upper bounds for the eigenvalues $\mu_j[\rho]$ of the polyharmonic operators with Neumann boundary conditions which are consistent with the power of $j$ in the Weyl's asymptotic formula (see \eqref{weyl}), among all densities which satisfy a suitable constraint. In particular, we consider two very natural constraints: $\int_{\Omega}\rho dx={\rm const.}$ and $\int_{\Omega}\rho^{\frac{N}{2m}}dx={\rm const.}$. This second constraint arises naturally since it is well-known (see e.g., \cite{lapidus}) that if we set $N(\mu):=\left\{\#\mu_j:\mu_j<\mu\right\}$, then $N(\mu)\sim\omega_N(2\pi)^{\frac{1}{N}}\mu^{\frac{N}{2m}}\int_{\Omega}\rho^{\frac{N}{2m}}dx$. This means that $\|\rho\|_{L^{\frac{N}{2m}}(\Omega)}$ describes the asymptotic distribution of the eigenvalues of problem \eqref{DENSITYPROBLEM} (and in particular implies the Weyl's law \eqref{weyl}). Most of the literature mentioned above considers only the fixed mass constraint. 

In view of the physical interpretation of problem \eqref{DENSITYPROBLEM} when $m=1$ and $N=1$ or $N=2$, it is very natural to ask whether it is possible to redistribute a fixed amount of mass on a string (of fixed length) or on a membrane (of fixed shape) such that all the eigenvalues become arbitrarily large when the body is left free to move, or, on the contrary, if there exists uniform upper bounds for all the eigenvalues. As highlithed in \cite{colbois_elsoufi}, uniform upper bounds with mass constraint exist if $N\geq 2$. In this paper, by using the techniques of \cite{colbois_elsoufi} we prove that if $N\geq 2m$, uniform upper bounds exist (see Theorem \ref{thm_upperNge4}), namely we prove that if $N\geq 2m$
\begin{equation}\label{SOLVED}
\mu_j[\rho]\leq C_{N,m}\frac{|\Omega|}{\int_{\Omega}\rho dx}\left(\frac{j}{|\Omega|}\right)^{\frac{2m}{N}},
\end{equation}
where $C_{N,m}$ depends only on $m$ and $N$. Surprisingly, in lower dimensions, uniform upper bounds do not hold. In fact we find explicit examples of densities with fixed mass and arbitrarily large eigenvalues (see Theorem \ref{counter23}). In this case, however, we are able to find upper bounds which depend also on $\|\rho\|_{L^{\infty}(\Omega)}$ (see Theorem \ref{sharpbounds_mass}), namely we prove that if $N<2m$
\begin{equation}\label{SOLVED2}
\mu_j[\rho]\leq C_{N,m}\frac{\|\rho\|_{L^{\infty}(\Omega)}^{\frac{2m}{N}-1}}{\left(\int_{\Omega}\rho dx\right)^{\frac{2m}{N}}}j^{\frac{2m}{N}},
\end{equation}
where again $C_{N,m}$ depends only on $m$ and $N$ and  the exponent of $\|\rho\|_{L^{\infty}(\Omega)}$ is sharp. We remark that this inequality holds when $m=1$ for $N=1$, and it is the analogue of the upper bounds \eqref{KR} proved by Krein \cite{krein} for the Dirichlet Laplacian on an interval (up to a universal constant). We note that in order to prove that  certain eigenvalue bounds under some natural constraints do not hold, one has to provide counterexamples. It is then natural to ask whether it is possible to find `weaker' bounds which include the correct quantities that explain  the counterexamples. This is the case of the bounds \eqref{SOLVED2}.

We note that the interplay between the dimension of the space and the order of the operator plays a crucial role in the existence of uniform upper bounds for the eigenvalues of problem \eqref{DENSITYPROBLEM} under mass constraint. We can summarize our first result in this way: 

 {\it ``If $N\geq 2m$ there exist uniform upper bounds with mass constraint for all the eigenvalues of \eqref{DENSITYPROBLEM}, while if $N<2m$ we can always redistribute a fixed amount of mass such that all the eigenvalues of \eqref{DENSITYPROBLEM} become arbitrarily large''}.

As for the the non-linear constraint $\int_{\Omega}\rho^{\frac{N}{2m}}dx={\rm const.}$, in view of the fact that $N(\mu)\sim\omega_N(2\pi)^{\frac{1}{N}}\mu^{\frac{N}{2m}}\int_{\Omega}\rho^{\frac{N}{2m}}dx$, it is natural to ask whether upper bounds of the form

\begin{equation}\label{OPEN}
\mu_j[\rho]\leq C_{N,m}\left(\frac{j}{\int_{\Omega}\rho^{\frac{N}{2m}}dx}\right)^{\frac{2m}{N}}
\end{equation}
hold. We will call bounds of the form \eqref{OPEN}  ``Weyl-type bounds''. Clearly, for $N=2m$ inequality \eqref{OPEN} is equivalent to \eqref{SOLVED}. For $N>2m$ we are able to find densities with fixed $L^{\frac{N}{2m}}$-norm and which produce arbitrarily large eigenvalues (see Theorem \ref{counterweyl23}). However, we are able to prove upper bounds for all the eigenvalues which involve both $\|\rho\|_{L^{\frac{N}{2m}}(\Omega)}$ and $\|\rho\|_{L^{\infty}(\Omega)}$ (see Theorem \ref{sharpbound_weyl}), namely we prove that if $N>2m$ then
\begin{equation}\label{PARTIAL}
\mu_j[\rho]\leq C_{N,m}\left(\frac{|\Omega|\|\rho\|_{L^{\infty}(\Omega)}^{\frac{N}{2m}}}{\int_{\Omega}\rho^{\frac{N}{2m}}dx}\right)^{1-\frac{2m}{N}}\left(\frac{j}{\int_{\Omega}\rho^{\frac{N}{2m}}dx}\right)^{\frac{2m}{N}},
\end{equation}
where $C_{N,m}$ depends only on $m$ and $N$ and the exponent of $\|\rho\|_{L^{\infty}(\Omega)}$ is sharp. Since \eqref{OPEN} holds for $N=2m$ we are led to conjecture that it must hold for any $N<2m$. We are still not able to prove \eqref{OPEN} for $N<2m$, and actually it seems to be a quite difficult issue. However we can prove the weaker inequality
\begin{equation}\label{PARTIAL2}
\mu_j[\rho]\leq C_{N,m}\left(\frac{|\Omega|\|\rho\|_{L^{\infty}(\Omega)}^{\frac{N}{2m}}}{\int_{\Omega}\rho^{\frac{N}{2m}}dx}\right)^{\frac{2m}{N}-1}\left(\frac{j}{\int_{\Omega}\rho^{\frac{N}{2m}}dx}\right)^{\frac{2m}{N}}.
\end{equation}
We leave the validity of \eqref{OPEN} for $N<2m$ as an open question. We refer to Remark \ref{openquestion} where we discuss relevant examples in support of the validity of our conjecture. In particular, we note that if \eqref{OPEN} holds true for $N<2m$, when $m=1$ and $N=1$ we would find uniform upper bounds for the eigenvalues of the Neumann Laplacian under the constraint that $\int_{\Omega}\sqrt{\rho}dx={\rm const.}$. We can summarize our second result as follows:

 {\it ``If $N>2m$, we can always find a density $\rho$ with fixed $\|\rho\|_{L^{\frac{N}{2m}}(\Omega)}$ such that all the eigenvalues of \eqref{DENSITYPROBLEM} are arbitrarily large, while we have uniform Weyl-type upper bounds when $N=2m$. We conjecture the existence of uniform Weyl-type upper bounds when $N<2m$''}.

We also mention \cite[\S\,9.2.3]{Henrot} where it is considered a spectral optimization problem for the Dirichlet Laplacian with the non-linear constraint $\int_{\Omega}\rho^p dx={\rm const.}$, where $p>N/2$ and $N\geq 2$ (see also \cite[\S\,5]{egorov}).


We have also considered the issue of lower bounds and we have found that `surprisingly' the interplay between the space dimension $N$ and the order $m$ of the operator plays a fundamental role also in the existence of lower bounds. In fact we are able to prove the following facts (see Theorems \ref{lowerbound_mass},\ref{lowerbound_counter_mass},\ref{lowerbound_weyl} and \ref{lowerbound_counter_weyl}):

{\it ``If $N<2m$ there exists a positive constant $C$ which depends only on $m,N$ and $\Omega$ such that the first positive eigenvalue of problem \eqref{DENSITYPROBLEM}  is bounded from below by $C\left(\int_{\Omega}\rho dx\right)^{-1}$, while if $N\geq 2m$, for all $j\in\mathbb N$ we can always redistribute a fixed amount of mass such that the first $j$ eigenvalues of \eqref{DENSITYPROBLEM} are arbitrarily close to zero''}

\noindent and

{\it ``If $N>2m$ there exists a positive constant $C$ which depends only on $m,N$ and $\Omega$ such that the first positive eigenvalue of problem \eqref{DENSITYPROBLEM}  is bounded from below by $C\|\rho\|_{L^{\frac{N}{2m}}(\Omega)}^{-1}$, while if $N\leq 2m$ for all $j\in\mathbb N$ we can always find densities with fixed $L^{\frac{N}{2m}}$-norm such that the first $j$ eigenvalues of \eqref{DENSITYPROBLEM} are arbitrarily close to zero''}.

We note that lower bounds for the first eigenvalue under one of the two constraints exist in the case that upper bounds with the same constraint do not exist. We remark that the situation is very different if we consider for example the issue of the minimization of the eigenvalues of \eqref{DENSITYPROBLEM} with $\rho\equiv 1$ among all bounded domains with fixed measure: it is standard to prove that there exist domains with fixed volume and such that the first $j$ eigenvalues can be made arbitrarily close to zero, in any dimension $N\geq 2$. 

Finally we remark that all the results of this paper can be adapted to the more general eigenvalue problem
$$
\mathcal L u=\mu\rho u,
$$
with Neumann boundary conditions, where $\mathcal L$ is defined by
$$
\mathcal Lu:=\sum_{0\leq |\alpha|,|\beta|\leq m}(-1)^{|\alpha|}\partial^{\alpha}(A_{\alpha\beta}\partial^{\beta}u)
$$
and is an elliptic operator of order $2m$, under suitable assumptions on the domain $\Omega$ and the coefficients of $A_{\alpha\beta}$. We refer to \cite{laproeurasian} for a detailed description of eigenvalue problems for higher order elliptic operators with density (see also \cite[\S\,7]{egorov}).

The present paper is organized as follows: Section \ref{preliminaries} is dedicated to some preliminaries. In Section \ref{MASS} we consider the problem of finding uniform upper bounds with mass constraint. In particular in Subsection \ref{sub:31} we prove uniform upper bounds \eqref{SOLVED} for $N\geq 2m$, in Subsection \ref{sub:33} we prove upper bounds \eqref{SOLVED2} for $N<2m$ and in Subsection \ref{sub:32} we provide counterexamples to uniform upper bounds in dimension $N<2m$. In Section \ref{WE} we investigate the existence of upper bounds with the non-linear constraint $\int_{\Omega}\rho^{\frac{N}{2m}}dx={\rm const.}$. In particular in Subsections \ref{sec:41}  and \ref{sub:43} we prove upper bounds \eqref{PARTIAL2} and \eqref{PARTIAL}, respectively, while in Subsection \ref{sec:42} we provide counterexamples to uniform upper bounds \eqref{OPEN} for $N>2m$. In Subsection \ref{sec:41} we state the open question whether bounds of the form \eqref{OPEN} hold if $N<2m$. In Section \ref{sec:lower} we consider lower bounds and in particular we discuss how the constraint, the space dimension and the order of the operator influence their existence. At the end of the paper we have two appendices, Appendix A and Appendix B. In Appendix A we discuss Neumann boundary conditions for higher order operators and develop some basic spectral theory for such operators. In Appendix B we prove some useful functional inequalities which are crucial in the proof of Theorem \ref{counter23} in Subsection \ref{sub:32}.


\section{Preliminaries and notation}\label{preliminaries}

Let $\Omega$ be a bounded domain (i.e., an open connected bounded set) of $\mathbb R^N$. By $H^m(\Omega)$ we shall denote the standard  Sobolev space of functions in $L^2(\Omega)$ with weak derivatives up to order $m$ in $L^2(\Omega)$, endowed with its standard norm defined by
\begin{equation*}
\|u\|_{H^m(\Omega)}:=\left(\int_{\Omega}|D^mu|^2+u^2 dx\right)^\frac{1}{2}
\end{equation*}
for all $u\in H^m(\Omega)$, where 
$$
|D^mu|^2:=\sum_{\substack{\alpha\in\mathbb N^N,\\|\alpha|= m}}|\partial^{\alpha}u|^2.
$$ 

In what follow we will use the standard multi-index notation. Hence, for $\alpha\in\mathbb N^N$, $\alpha=(\alpha_1,...,\alpha_N)$, we shall denote by $|\alpha|$ the quantity $|\alpha|=\alpha_1+\cdots+\alpha_N$. Moreover, for $\alpha,\beta\in\mathbb N^N$, $\alpha+\beta=(\alpha_1+\beta_1,...,\alpha_N+\beta_N)$ and $\alpha!=\alpha_1!\cdots\alpha_N!$. For $x\in\mathbb R^N$, $x=(x_1,...,x_N)$, we will write $x^{\alpha}=x_1^{\alpha_1}\cdots x_N^{\alpha_N}$. For a function $u$ of class $C^{|\alpha|}$, we write $\partial^{\alpha}u=\frac{\partial^{|\alpha|}u}{\partial_{x_1}^{\alpha_1}\cdots\partial_{x_N}^{\alpha_N}}$. Finally, for a function $U:\mathbb R\rightarrow\mathbb R$ and $l\in\mathbb N$, we shall write $U^{(l)}(x)$ to denote the $l$-th derivative of $U$ with respect to $x$.

In the sequel we shall assume that the domain $\Omega$ is such that the embedding of $H^m(\Omega)$ into $L^2(\Omega)$ is compact (which is ensured, for example, if $\Omega$ is a bounded domain with Lipschitz boundary). By $\mathcal R$ we shall denote the subset of $L^{\infty}(\Omega)$ of those functions $\rho\in L^{\infty}(\Omega)$ such that ${\rm ess}\inf_{\Omega}\rho>0$. 

We shall consider the following eigenvalue problem:
\begin{equation}\label{nweak}
\int_{\Omega}D^m u:D^m \varphi dx=\mu\int_{\Omega}\rho u \varphi dx\,,\ \ \ \forall\varphi\in H^m(\Omega)
\end{equation}
in the unknowns $u\in H^m(\Omega)$ (the eigenfunction), $\mu\in\mathbb R$ (the eigenvalue), where 
$$
D^m u:D^m\varphi:=\sum_{|\alpha|=m}{\partial^{\alpha} u}{\partial^{\alpha}\varphi}.
$$

We note that problem \eqref{nweak} is the weak formulation of the following eigenvalue problem:
\begin{equation}\label{nclassic}
\begin{cases}
(-\Delta)^m u=\mu\rho u,&{\rm in\ }\Omega,\\
\mathcal N_0 u =\cdots=\mathcal N_{m-1} u =0,&{\rm on\ }\partial\Omega,
\end{cases}
\end{equation}
in the unknowns $u\in C^{2m}(\Omega)\cap C^{2m-1}(\overline\Omega)$ and $\mu\in\mathbb R$. Here $\mathcal N_j$ are uniquely defined `complementing' boundary operators of degree at most $2m-1$ (see \cite{gazzola} for details), which we will call {\it Neumann boundary conditions} (see Appendix \ref{app:neumann}).

\begin{example}
If $m=1$, $\mathcal N_0u=\frac{\partial u}{\partial\nu}$ and \eqref{nclassic} is the classical formulation of the Neumann eigenvalue problem for the Laplace operator, namely
\begin{equation}\label{nclassic1}
\begin{cases}
-\Delta u=\mu\rho u,&{\rm in\ }\Omega,\\
\frac{\partial u}{\partial\nu}=0,&{\rm on\ }\partial\Omega,
\end{cases}
\end{equation}
in the unknowns $u\in C^{2}(\Omega)\cap C^{1}(\overline\Omega)$ and $\mu\in\mathbb R$, while if $m=2$ we have the Neumann eigenvalue problem for the biharmonic operator, namely
\begin{equation}\label{nclassic2}
\begin{cases}
\Delta^2u=\mu\rho u,&{\rm in\ }\Omega,\\
\frac{\partial^2u}{\partial\nu^2}=0,&{\rm on\ }\partial\Omega,\\
{\rm div}_{\partial\Omega}(D^2u\cdot\nu)+\frac{\partial\Delta u}{\partial\nu}=0,&{\rm on\ }\partial\Omega,
\end{cases}
\end{equation}
in the unknowns $u\in C^{4}(\Omega)\cap C^{3}(\overline\Omega)$ and $\mu\in\mathbb R$. Here ${\rm div}_{\partial\Omega}$ denotes the tangential divergence operator on $\partial\Omega$ (we refer to \cite[\S\,7]{shapes} for more details on tangential operators). 
\end{example}

In Appendix \ref{app:neumann} we discuss in more detail boundary conditions for problems \eqref{nclassic} and \eqref{nclassic2} and, more in general, Neumann boundary conditions for the polyharmonic operators.


It is standard to prove (see Theorem \ref{teorema_neumann_883}) that the eigenvalues of \eqref{nweak} are non-negative, have finite multiplicity and consist of a sequence diverging to $+\infty$ of the form
$$
0=\mu_1[\rho]=\cdots=\mu_{d_{N,m}}[\rho]<\mu_{d_{N,m}+1}[\rho]\leq\cdots\leq\mu_j[\rho]\leq\cdots\nearrow+\infty,
$$
where
$$
d_{N,m}:=\binom{N+m-1}{N}.
$$
The eigenfunctions associated with the eigenvalue $\mu=0$ are the polynomials of degree at most $m-1$ in $\mathbb R^N$ (the dimension of the space spanned by the polynomials of degree at most $m-1$ in $\mathbb R^N$ is exactly $d_{N,m}$). We note that we have highlithed the dependence of the eigenvalues upon the density $\rho$, which is the main object of study of the present paper.

By  standard spectral theory, we deduce the validity of the following variational representation of the eigenvalues (see \cite[\S\,IV]{cohil} for more details):

\begin{thm}
Let $\Omega$ be a bounded domain in $\mathbb R^N$ such that the embedding $H^m(\Omega)\subset L^2(\Omega)$ is compact. Let $\rho\in \mathcal R$. Then for all $j\in\mathbb N$ we have
\begin{equation}\label{minimax212}
\mu_{j}[\rho]=\inf_{\substack{V\leq H^m(\Omega)\\{\rm dim}V=j}}\sup_{\substack{0\ne u\in V}}\frac{\int_{\Omega}|D^mu|^2dx}{\int_{\Omega}\rho u^2dx}.
\end{equation}
\end{thm}

We conclude this section by recalling the asymptotic behavior of the eigenvalues as $j\rightarrow +\infty$, which is described by the Weyl's law.

\begin{thm}\label{weylthm}
Let $\Omega$ be a bounded domain in $\mathbb R^N$ with Lipschitz boundary. Let $\rho\in \mathcal R$. Then
\begin{equation}\label{weyl}
\mu_j[\rho]\sim\frac{(2\pi)^{2m}}{\omega_N^{\frac{2m}{N}}}\left(\frac{j}{\int_{\Omega}\rho^{\frac{N}{2m}}}\right)^{\frac{2m}{N}}
\end{equation}
as $j\rightarrow +\infty$.
\end{thm}
We refer to \cite{lapidus} for a proof of Theorem \ref{weylthm}.


\section{Upper bounds with mass constraint}\label{MASS}

In this section we consider the problem of finding uniform upper bounds for the $j$-th eigenvalue $\mu_j[\rho]$ among all mass densities $\rho\in\mathcal R$ which preserve the mass (that is, among all $\rho\in\mathcal R$ such that $\int_{\Omega}\rho dx={\rm const.}$), and which show the correct growth in the power of $j$ with respect to the Weyl's law \eqref{weyl}. In particular, in Subsection \ref{sub:31} we prove that such bounds exist if $N\geq 2m$ (see Theorem \ref{thm_upperNge4}), while in Subsection \ref{sub:32} we will give counter-examples in dimension $N<2m$ (see Theorem \ref{counter23}). Moreover, in Subsection \ref{sub:33} we establish upper bounds in the case $N<2m$ which involve also a suitable power of $\|\rho\|_{L^{\infty}(\Omega)}$ (see Theorem \ref{sharpbounds_mass}) which turns out to be sharp.


\subsection{Uniform upper bounds with mass constraint for \texorpdfstring{$N\geq 2m$}{N>=2m}}\label{sub:31}

In this subsection we will prove the existence of uniform upper bounds for $N\ge 2m$ with respect to mass preserving densities.

The main tool which we will use is a result of decomposition of a metric measure space by annuli (see \cite[Theorem\,1.1]{gny}). We recall it here for the reader's convenience.

Let $(X,d)$ be a metric space. By an annulus in $X$ we mean any set $A\subset X$ of the form
\begin{equation*}
A=A(a,r,R)=\left\{x\in X:r<d(x,a)<R\right\},
\end{equation*}
where $a\in X$ and $0\leq r<R<+\infty$. By $2A$ we denote 
\begin{equation*}
2A=2A(a,r,R)=\left\{x\in X:\frac{r}{2}<d(x,a)<2R\right\}.
\end{equation*}
We are ready to state the following theorem (see \cite[Theorem\,11]{gny}):
\begin{thm}\label{gny}
Let $(X,d)$ be a metric space and $\nu$ be a Radon measure on it. Assume that the following properties are satisfied:
\begin{enumerate}[i)]
\item there exists a constant $\Gamma$ such that any metric ball of radius $r$ can be covered by at most $\Gamma$ balls of radius $r/2$;
\item all metric balls in $X$ are precompact sets;
\item the measure $\nu$ is non-atomic.
\end{enumerate}
Then for any integer $j$ there exist a sequence $\left\{A_i\right\}_{i=1}^j$ of $j$ annuli in $X$ such that, for any $i=1,...,j$
\begin{equation*}
\nu(A_i)\geq c\frac{\nu(X)}{j},
\end{equation*}
and the annuli $2A_i$ are pairwise disjoint. The constant $c$ depends only on $\Gamma$.
\end{thm}

In the sequel we will need also the following corollary of Theorem \ref{gny} (see \cite[Remark\,3.13]{gny}):

\begin{corollary}\label{corollarygny0}
Let the assumptions of Theorem \ref{gny} hold. If in addition $0<\nu(X)<\infty$, each annulus $A_i$ has either internal radius $r_i$ such that
\begin{equation}\label{gny-rad-est}
r_i\geq\frac{1}{2}\inf\left\{r\in\mathbb R:V(r)\geq v_j\right\},
\end{equation}
where $V(r):=\sup_{x\in X}\nu(B(x,r))$ and $v_j=c\frac{\nu(X)}{j}$ , or is a ball of radius $r_i$ satisfying \eqref{gny-rad-est}.
\end{corollary}

We are now ready to state the main result of this section.

\begin{thm}\label{thm_upperNge4}
Let $\Omega$ be a bounded domain in $\mathbb R^N$, $N\geq 2m$, such that the embedding $H^m(\Omega)\subset L^2(\Omega)$ is compact. Let $\rho\in\mathcal R$. Then for every $j\in\mathbb N$ we have
\begin{equation}\label{upperNge4}
\mu_j[\rho]\leq C_{N,m}\frac{1}{\int_{\Omega}\rho dx |\Omega|^{-1}}\left(\frac{j}{|\Omega|}\right)^{\frac{2m}{N}},
\end{equation}
where $C_{N,m}$ is a constant which depends only on $N$ and $m$.
\end{thm}

\begin{rem}
Inequality \eqref{upperNge4} says that there exists a uniform upper bound for all the eigenvalues $\mu_j[\rho]$ with respect to those densities $\rho\in\mathcal R$ which give the same mass $M=\int_{\Omega}\rho dx$. We note that the quantity $\int_{\Omega}\rho dx |\Omega|^{-1}=M/|\Omega|$ is an average density, i.e., the total mass over the total volume of $\Omega$. Moreover, from \eqref{upperNge4} it follows that
\begin{equation*}
\mu_j[\rho]\leq C_{N,m}\left(\frac{j}{|\Omega|}\right)^{\frac{2m}{N}},
\end{equation*}
for all densities $\rho\in\mathcal R$ and bounded domains $\Omega$ (with $H^m(\Omega)\subset L^2(\Omega)$ compact) such that $\int_{\Omega}\rho dx=|\Omega|$.
\end{rem}

\begin{proof}[Proof of Theorem \ref{thm_upperNge4}]
The proof is based on the general method described by Grigor'yan, Netrusov and Yau in \cite{gny} (see Theorem \ref{gny}; see also \cite{colbois_elsoufi,colbois} for the case of the Laplace operator). In particular, we will build a suitable family of disjointly supported test functions with controlled Rayleigh quotient.

Let $0\leq r<R<+\infty$. Let $U_{r,R}:[0,+\infty[\rightarrow[0,1]$ be defined by
\begin{equation}\label{radialu}
U_{r,R}(t):=\begin{cases}
0 & {\rm if\ } t\in[0,r/2[\cup[2R,+\infty[,\\
\sum_{i=0}^{2m-1}\frac{a_i}{r^i} t^i & {\rm if\ } t\in[r/2,r[,\\
1 & {\rm if\ } t\in[r,R[,\\
\sum_{i=0}^{2m-1}\frac{b_i}{R^i} t^i & {\rm if\ } t\in[R,2R[.
\end{cases}
\end{equation}
The coefficients $\left\{a_i\right\}_{i=0}^{2m-1}$, $\left\{b_i\right\}_{i=0}^{2m-1}$ are uniquely determined by the equations
\begin{equation}\label{system}
\begin{cases}
U_{r,R}(r/2)=0, & \\
U_{r,R}(r)=1, &\\
U_{r,R}(R)=1,&\\
U_{r,R}(2R)=0,&\\
U_{r,R}^{(l)}(r)=U_{r,R}^{(l)}(r/2)=U_{r,R}^{(l)}(R)=U_{r,R}^{(l)}(2R)=0,& \forall l=1,...,m-1.
\end{cases}
\end{equation}
We note that \eqref{system} can be written as
\begin{equation}\label{system2}
\begin{cases}
\sum_{i=0}^{2m-1}\frac{a_i}{2^i}=0, & \\
\sum_{i=0}^{2m-1}{a_i}=1, &\\
\sum_{i=0}^{2m-1}{b_i}=1, & \\
\sum_{i=0}^{2m-1}{2^i b_i}=0, &\\
\sum_{i=l}^{2m-1}i(i-1)\cdots(i-l+1)2^{l-i}a_i=0,& \forall l=1,...,m-1,\\
\sum_{i=l}^{2m-1}i(i-1)\cdots(i-l+1)a_i=0,& \forall l=1,...,m-1,\\
\sum_{i=l}^{2m-1}i(i-1)\cdots(i-l+1)2^{i-l}b_i=0,& \forall l=1,...,m-1,\\
\sum_{i=l}^{2m-1}i(i-1)\cdots(i-l+1)b_i=0,& \forall l=1,...,m-1,
\end{cases}
\end{equation}
which is a system of $4m$ equations in $4m$ unknowns $a_0,...,a_{2m-1}$, $b_0,...,b_{2m-1}$. It is standard to see that \eqref{system2} admits an unique non-zero solution. Moreover, the coefficients $a_0,...,a_{2m-1}$, $b_0,...,b_{2m-1}$ depend only on $m$.

We note that by construction $U_{r,R}\in C^{m-1}([0,+\infty[)\cap C^{m-1,1}([0,2R])$. Let now $A=A(a,r,R)$ be an annulus in $\mathbb R^N$. We define a function $u_{a,r,R}$ supported on $2A$ and such that $u_{a,r,R}\leq 1$ on $2A$ and $u_{a,r,R}\equiv 1$ on $A$ by setting
\begin{equation}\label{radialu2}
u_{a,r,R}(x):=U_{r,R}(|x-a|).
\end{equation}
By construction, the restriction of this function to $\Omega$  belongs to the Sobolev space $H^m(\Omega)$. Now we exploit Theorem \ref{gny} with $X=\Omega$ endowed with the Euclidean distance, and the measure $\nu$ given by $\nu(E):=\int_{E} \rho dx$ for all measurable $E\subset\Omega$. The hypothesis of Theorem \ref{gny} are clearly satisfied. Hence, for each index $j\in\mathbb N$ we find $2j$ annuli $\left\{A_i\right\}_{i=1}^{2j}$ such that $2A_i$ are disjoint and  
\begin{equation*}
\int_{A_i\cap\Omega}\rho dx\geq c_N\frac{\int_{\Omega}\rho dx}{2j},
\end{equation*}
where $c_N>0$ depends only on $N$. Since we have $2j$ annuli $2A_i=2A_i(a_i,r_i,R_i)$, $i=1,...,2j$, we can choose $j$ of such annuli, say $\left\{2A_{i_1},...,2A_{i_j}\right\}$ such that 
\begin{equation}\label{choice}
|2A_{i_k}|\leq\frac{|\Omega|}{j}
\end{equation}
for all $k=1,...,j$. To each of such annuli, we associate a function $u_{i_k}$ defined by
\begin{equation}\label{radialu3}
u_{i_k}(x):=u_{a_{i_k},r_{i_k},R_{i_k}}(x).
\end{equation}
We have then built a family of $j$ disjointly supported functions, which we relabel as $u_1,...,u_j$, and whose restriction to $\Omega$ belong to the space $H^m(\Omega)$.

From the min-max principle \eqref{minimax212} it follows that
\begin{equation}\label{minmax1}
\mu_{j}[\rho]\leq\max_{\substack{0\ne u\in V_j}}\frac{\int_{\Omega}|D^mu|^2dx}{\int_{\Omega}\rho u^2dx},
\end{equation}
where $V_j$ is the subspace of $H^m(\Omega)$ generated by $u_1,...,u_j$ (and which has dimension $j$). Since the space is generated by $j$ disjointly supported functions, it is standard to prove that \eqref{minmax1} is equivalent to the following:
\begin{equation}\label{minmax2}
\mu_{j}[\rho]\leq\max_{i=1,...,j}\frac{\int_{\Omega}|D^m u_i|^2dx}{\int_{\Omega}\rho u_i^2dx}.
\end{equation}
This means that it is sufficient to have a control on the Rayleigh quotient of each of the generating functions $u_i$ in order to bound $\mu_j[\rho]$. It is standard to see that, if $u\in C^m(\Omega)$ is given by $u(x)=U(|x-a|)$, for some function $U$ of one real variable, then for all $\alpha\in\mathbb N^N$ with $|\alpha|=m$
\begin{equation}\label{m-der}
{\partial^{\alpha}u(x)}=\sum_{k=1}^m\frac{c_{N,k,\alpha}}{|x-a|^{m-k}}U^{(k)}(|x-a|),
\end{equation}
where $c_{N,k,\alpha}$ depends only on $N$, $k$ and $\alpha$. From \eqref{m-der} it follows then that there exists a constant $C_{N,m}>0$ which depends only on $m$ and $N$ such that
\begin{equation}\label{m-der-est}
|D^m u(x)|^2\leq C_{N,m}\sum_{k=1}^m\frac{(U^{(k)}(x-a))^2}{|x-a|^{2(m-k)}}.
\end{equation}
Through the rest of the proof we will denote by $C_{N,m}$ a positive constant which depends only on $m$ and $N$ and which can be eventually re-defined line by line.

By standard approximation of functions in $H^m(\Omega)$ by smooth functions (see \cite[\S\,5.3]{evans}), from \eqref{m-der-est} we deduce that if $u\in H^m(\Omega)$ is given by $u(x)=U(|x-a|)$, for some function $U$ of one real variable, then
\begin{equation}\label{hessianintegral}
\int_{\Omega}|D^mu|^{2p}dx\leq C_{N,m}\int_{\Omega}\sum_{k=1}^m\frac{(U^{(k)}(|x-a|))^{2p}}{|x-a|^{2p(m-k)}} dx,
\end{equation}
for all $p>0$. We are now ready to estimate the right-hand side of \eqref{minmax2}. For the denominator we have
\begin{equation}\label{den}
\int_{\Omega}\rho u_i^2dx=\int_{2A_i\cap\Omega}\rho u_i^2dx\geq\int_{A_i\cap\Omega}\rho u_i^2dx=\int_{A_i\cap\Omega}\rho dx\geq c_N\frac{\int_{\Omega}\rho dx}{2j}.
\end{equation}
This follows from the fact that $u_i\leq 1$ and $u_i\equiv 1$ on $A_i$ and from Theorem \ref{gny}. For the numerator, since $N\geq 2m$, we have
\begin{multline}\label{step1}
\int_{\Omega}|D^mu_i|^2dx=\int_{\Omega\cap 2A_i}|D^mu_i|^2dx\leq\int_{2A_i}|D^mu_i|^2dx\\
\leq\left(\int_{2A_i}|D^mu_i|^{N/m}dx\right)^{2m/N}|2A_i|^{1-2m/N}.
\end{multline}
From \eqref{choice}, we have that $|2A_i|^{1-2m/N}\leq\left({|\Omega|}/{j}\right)^{1-2m/N}$. Moreover, from \eqref{radialu}, \eqref{hessianintegral} and standard calculus we have that
\begin{equation}\label{step111}
\left(\int_{2A_i}|D^mu_i|^{N/2}dx\right)^{2m/N}\leq C_{N,m}.
\end{equation}
From \eqref{den}, \eqref{step1} and \eqref{step111} we have that
\begin{equation}\label{step222}
\frac{\int_{\Omega}|D^m u_i|^2dx}{\int_{\Omega}\rho u_i^2dx}\leq C_{N,m}\frac{1}{\int_{\Omega}\rho dx|\Omega|^{-1}}\left(\frac{j}{|\Omega|}\right)^{\frac{2m}{N}}
\end{equation}
for all $i=1,...,j$. From \eqref{minmax2} and \eqref{step222} we deduce the validity of \eqref{upperNge4}. This concludes the proof.
\end{proof}


\subsection{Upper bounds with mass constraint for \texorpdfstring{$N<2m$}{N<2m}}\label{sub:33}

We note that the proof of Theorem \ref{thm_upperNge4} does not work in the case $N<2m$. Indeed, as we will see in Subsection \ref{sub:32} (see Theorem \ref{counter23}), bounds of the form \eqref{upperNge4} do not hold if $N<2m$. In this subsection we prove upper bounds for the eigenvalues $\mu_j[\rho]$ which involve also a suitable power of $\|\rho\|_{L^{\infty}(\Omega)}$, namely, we prove the following theorem:

\begin{thm}\label{sharpbounds_mass}
Let $\Omega$ be a bounded domain in $\mathbb R^N$, $N<2m$, such that the embedding $H^m(\Omega)\subset L^2(\Omega)$ is compact. Let $\rho\in\mathcal R$. Then for every $j\in\mathbb N$ we have
\begin{equation}\label{sharp_bounds}
\mu_j[\rho]\leq C_{N,m} \frac{\|\rho\|_{L^{\infty}}^{\frac{2m}{N}-1}}{\left(\int_{\Omega}\rho dx\right)^{\frac{2m}{N}}}j^{\frac{2m}{N}},
\end{equation}
where $C_{N,m}>0$ depends only on $m$ and $N$.

\begin{rem}\label{remKR}
We recall the following well-known result by Krein \cite{krein} which states that in the case of the equation $-u''(x)=\mu \rho(x) u(x)$ on $]0,l[$ with Dirichlet boundary conditions, we have
\begin{equation}\label{KR}
\mu_j[\rho]\leq\frac{\pi^2\|\rho\|_{L^{\infty}(]0,l[)}}{\left(\int_0^l\rho dx\right)^2}j^2,
\end{equation}
which is the analogous of \eqref{sharp_bounds} for the Laplace operator ($m=1$) in dimension $N=1$. Actually, inequality \eqref{KR} is sharp, i.e., for all $j\in\mathbb N$ there exists $\rho_j\in\mathcal R$ such that the equality holds in \eqref{KR} when $\rho=\rho_j$.
\end{rem}

\begin{proof}[Proof of Theorem \ref{sharpbounds_mass}]
In order to prove \eqref{sharp_bounds} we exploit in more detail Theorem \ref{gny} and Corollary \ref{corollarygny0}. As in the proof of Theorem \ref{thm_upperNge4}, for any measurable $E\subset\Omega$, let $\nu(E):=\int_{E}\rho dx$. By following the same lines of the proof of Theorem \ref{thm_upperNge4}, we find for each $j\in\mathbb N$, $j$ annuli $A_1,...,A_j$  such that  the annuli $2A_i$ are disjoint, $\int_{A_i\cap\Omega}\rho dx\geq c_N\frac{\int_{\Omega}\rho dx}{2j}$ for all $i=1,...,j$, where $c_N>0$ depends only on $N$, and moreover $|2A_i|\leq{|\Omega|}/{j}$ for all $i=1,...,j$.

Let $r_i$ and $R_i$ denote the inner and outer radius of $A_i$, respectively ($r_i$ denotes the radius of $A_i$ if $A_i$ is a ball). Associated to each annulus $A_i$ we construct a test function $u_i$ supported on $2A_i$ and such that $u_i\equiv 1$ on $A_i$, and which satisfies
$$
\int_{2A_i\cap\Omega}\rho u_i^2 dx\geq\int_{A_i\cap\Omega}\rho dx\geq c_N\frac{\int_{\Omega}\rho dx}{2j}.
$$
and
$$
\int_{2A_i}|D^mu_i|^2dx\leq C_{N,m}(R_i^{N-2m}+r_i^{N-2m})\leq 2C_{N,m} r_i^{N-2m},
$$
where $C_{N,m}>0$ depends only on $m$ and $N$ (see \eqref{radialu}, \eqref{radialu2}, \eqref{radialu3} and \eqref{hessianintegral}). In what follows we shall denote by $C_{N,m}$ a positive constant which depends only on $m$ and $N$ and which can be eventually re-defined line by line. Then, if $A_i$ is an annulus of inner radius $r_i$ or a ball of radius $r_i$, we have
\begin{equation}\label{rayestimate}
\frac{\int_{\Omega}|D^m u_i|^2dx}{\int_{\Omega}\rho u_i^2dx}\leq C_{N,m} r_i^{N-2m}\frac{j}{\int_{\Omega}\rho dx},
\end{equation}
for all $i=1,...,j$.

We note that Corollary \ref{corollarygny0} provides an estimate for the inner radius of the annuli given by the decomposition of Theorem \ref{gny} (and, respectively, an estimate of the radius of the ball in the case that the decomposition of the space produces a ball). In particular, if $r_i$ is the inner radius of $A_i$, we have for all $i=1,...,j$ 
\begin{equation*}
r_i\geq\frac{1}{2}\inf\left\{r\in\mathbb R:V(r)\geq v_j\right\},
\end{equation*}
where $V(r):=\sup_{x\in\Omega}\int_{B(x,r)}\rho dx$ and $v_j=c_N\frac{\int_{\Omega}\rho dx}{2 j}$. Let $B_j:=\left\{r\in\mathbb R:V(r)>v_j\right\}$. If $r\in B_j$, then 
$$
c_N\frac{\int_{\Omega}\rho dx}{2j}=v_j<\sup_{x\in\Omega}\int_{B(x,r)}\rho dx\leq\omega_N r^N\|\rho\|_{L^{\infty}(\Omega)},
$$
where $\omega_N$ denotes the volume of the unit ball in $\mathbb R^N$. This means that
$$
r^N> c_N\frac{\int_{\Omega}\rho dx}{2j\omega_N\|\rho\|_{L^{\infty}(\Omega)}}
$$
for all $r\in B_j$. Hence, if $r_0:=\inf B_j$, then $r_0^N\geq c_N\frac{\int_{\Omega}\rho dx}{2j\omega_N\|\rho\|_{L^{\infty}(\Omega)}}$ and therefore
\begin{equation}\label{radiusestimate}
r_i^N\geq c_N\frac{\int_{\Omega}\rho dx}{2^{N+1}j\omega_N\|\rho\|_{L^{\infty}(\Omega)}}
\end{equation}
for all $i=1,...,j$. From \eqref{rayestimate} and \eqref{radiusestimate} it follows that
$$
\frac{\int_{\Omega}|D^m u_i|^2dx}{\int_{\Omega}\rho u_i^2dx}\leq C_{N,m} \frac{\|\rho\|_{L^{\infty}(\Omega)}^{\frac{2m}{N}-1}}{\left(\int_{\Omega}\rho dx\right)^{\frac{2m}{N}}}j^{\frac{2m}{N}},
$$
for all $i=1,...,j$, which implies \eqref{sharp_bounds} by \eqref{minimax212} and by the fact that $u_i$ are disjointly supported (see also \eqref{minmax1} and \eqref{minmax2}). This concludes the proof.

\end{proof}
\end{thm}


\subsection{Non-existence of uniform upper bounds for \texorpdfstring{$N<2m$}{N<2m} and sharpness of the exponent of \texorpdfstring{$\|\rho\|_{L^{\infty}(\Omega)}$}{Linf} in \texorpdfstring{\eqref{sharp_bounds}}{3}}\label{sub:32}
In this subsection we will prove that if $N<2m$, there exist families $\left\{\rho_{\varepsilon}\right\}_{\varepsilon\in]0,\varepsilon_0[}\subset\mathcal R$ such that $\int_{\Omega}\rho_{\varepsilon}dx\rightarrow\omega_N$ as $\varepsilon\rightarrow 0^+$ and $\mu_j[\rho_{\varepsilon}]\rightarrow +\infty$ for all $j\geq d_{N,m}+1$ as $\varepsilon\rightarrow 0^+$, and moreover we will provide the rate of divergence to $+\infty$ of the eigenvalues with respect to the parameter $\varepsilon$. This means that in dimension $N<2m$ we can redistribute a bounded amount of mass in $\Omega$ in such a way that all the positive eigenvalues become arbitrarily large. This is achieved, for example, by concentrating all the mass at one point of $\Omega$. Moreover, the families $\left\{\rho_{\varepsilon}\right\}_{\varepsilon\in]0,\varepsilon_0[}$ considered in this subsection will provide the sharpness of the power of $\|\rho\|_{L^{\infty}(\Omega)}$ in \eqref{sharp_bounds}.

Through all this subsection, $\Omega$ will be a bounded domain in $\mathbb R^N$ with Lipschitz boundary. Assume without loss of generality that $0\in\Omega$ and let $\varepsilon_0\in]0,1[$ be such that $B(0,\varepsilon)\subset\subset\Omega$ for all $\varepsilon\in]0,\varepsilon_0[$ (all the result of this section hold true if we substitute $0$ with any other $x_0\in\Omega$). For all $\varepsilon\in]0,\varepsilon_0[$ let $\rho_\varepsilon\in\mathcal R$ be defined by
\begin{equation}\label{rhoeps}
\rho_{\varepsilon}:=\varepsilon^{2m-N-\delta}+\varepsilon^{-N}\chi_{B(0,\varepsilon)},
\end{equation}
for some $\delta\in]0,1/2[$, which we fix once for all and which can be chosen arbitrarily close to zero. We have the following theorem:

\begin{thm}\label{counter23}
Let $\Omega$ be a bounded domain in $\mathbb R^N$, $N<2m$, with Lipschitz boundary. Let $\rho_{\varepsilon}\in\mathcal R$ be defined by \eqref{rhoeps} for all $\varepsilon\in]0,\varepsilon_0[$. Then
\begin{enumerate}[i)]
\item $\lim_{\varepsilon\rightarrow 0^+}\int_{\Omega}\rho_{\varepsilon}dx=\omega_N$;
\item for all $j\in\mathbb N$, $j\geq d_{N,m}+1$, there exists $c_j>0$ which depends only on $m$, $N$, $\Omega$ and $j$ such that (up to subsequences) $\lim_{\varepsilon\rightarrow 0^+}\mu_j[\rho_{\varepsilon}]\varepsilon^{2m-N-\delta}= c_j$ (for all $\delta\in]0,1/2[$).
\end{enumerate}
\end{thm}

From Theorem \ref{counter23} it follows that for $N<2m$ and $j\geq d_{N,m}+1$
$$
\lim_{\varepsilon\rightarrow 0^+}\mu_j[\rho_{\varepsilon}]= +\infty.
$$
Moreover, since $\|\rho_{\varepsilon}\|_{L^{\infty}(\Omega)}=\varepsilon^{-N}$, it follows that the power of $\|\rho\|_{L^{\infty}(\Omega)}$ in \eqref{sharp_bounds} is sharp.

We remark that the proof of Theorem \ref{counter23} requires some precise inequalities for function in $H^m(\Omega)$ and $N<2m$ which we prove in Lemmas \ref{pre2} and \ref{meanlemma} of the Appendix \ref{app:func}.

\begin{proof}[Proof of Theorem \ref{counter23}]
The proof of point $i)$ is a standard computation. In fact
\begin{equation*}
\lim_{\varepsilon\rightarrow 0^+}\int_{\Omega}\rho_{\varepsilon}dx=\lim_{\varepsilon\rightarrow 0^+}\varepsilon^{2m-N-\delta}|\Omega|+\lim_{\varepsilon\rightarrow 0^+}\varepsilon^{-N}|B(0,\varepsilon)|=\omega_N.
\end{equation*}

We prove now $ii)$. In order to simplify the notation, from now on we will denote an eigenvalue $\mu_j[\rho_{\varepsilon}]$ simply as $\mu_j[\varepsilon]$. The proof of $ii)$ is divided into two steps. In the first step we will prove that there exists a positive constant $c_j^{(1)}>0$ which depends only on $m$, $N$, $j$ and $\Omega$ such that $\mu_j[\varepsilon]\leq c_j^{(1)}\varepsilon^{N-2m+\delta}$. In the second step we will prove that there exists a positive constant $c_j^{(2)}>0$ which depends only on $m$, $N$, $j$ and $\Omega$ such that $\mu_j[\varepsilon]\geq c_j^{(2)}\varepsilon^{N-2m+\delta}$. This yields the result (up to choosing a suitable subsequence of $\left\{\rho_{\varepsilon}\right\}_{\varepsilon\in]0,\varepsilon_0[}$).

{\it Step 1.} We note that $\rho_{\varepsilon}\geq\varepsilon^{2m-N-\delta}$, hence for all $u\in H^m(\Omega)$,
\begin{equation}\label{step2-explosion}
\frac{\int_{\Omega}|D^mu|^2dx}{\int_{\Omega}\rho_{\varepsilon}u^2dx}\leq\varepsilon^{N-2m+\delta}\frac{\int_{\Omega}|D^mu|^2dx}{\int_{\Omega}u^2dx}.
\end{equation}
By taking the minimum and the maximum into \eqref{step2-explosion}, by \eqref{minimax212} we have that $\mu_j[\varepsilon]\leq\varepsilon^{N-2m+\delta}\mu_j[1]$ for all $j\in\mathbb N$. Hence $c_j^{(1)}=\mu_j[1]$.

{\it Step 2.} We introduce the function $\tilde\rho_{\varepsilon}$ defined by
\begin{equation}\label{tilderho}
\tilde\rho_{\varepsilon}:=1+\varepsilon^{-2m+\delta}\chi_{B(0,\varepsilon)}.
\end{equation}
We note that $\mu_j[\varepsilon]$ is an eigenvalue of \eqref{nweak} with $\rho=\rho_{\varepsilon}$ if and only if  $\tilde\mu_j[\varepsilon]:=\varepsilon^{2m-N-\delta}\mu_j[\varepsilon]$ is an eigenvalue of problem \eqref{nweak} with $\rho=\tilde\rho_{\varepsilon}$, where $\tilde\rho_{\varepsilon}$ is defined by \eqref{tilderho}.

We prove now that for all $j\geq d_{N,m}+1$ there exist $c>0$ which depends only on $m,n$ and $\Omega$ such that $\tilde\mu_j[\varepsilon]\geq c$ for all $\varepsilon\in]0,\varepsilon_0[$. This implies the existence of constants $c_j^{(2)}>0$ such that $\mu_j[\varepsilon]\geq c_j^{(2)}\varepsilon^{N-2m+\delta}$.

We recall that the first positive eigenvalue is $\tilde\mu_{d_{N,m}+1}[\varepsilon]$. From the min-max principle \eqref{minimax212} we have
\begin{equation*}
\tilde\mu_{d_{N,m}+1}[\varepsilon]=\inf_{u\in V_{\varepsilon}}\frac{\int_{\Omega}|D^mu|^2dx}{\int_{\Omega}u^2dx+\varepsilon^{-2m+\delta}\int_{B(0,\varepsilon)}u^2dx},
\end{equation*}
where 
$$
V_{\varepsilon}:=\left\{u\in H^m(\Omega):\int_{\Omega}\tilde\rho_{\varepsilon}u(x)x^{\alpha}dx=0,\ \forall\,0\leq|\alpha|\leq m-1\right\}.
$$
We argue by contradiction. Assume that $\tilde\mu_{d_{N,m}+1}[\varepsilon]\rightarrow 0$ as $\varepsilon\rightarrow 0^+$. Let $u_{\varepsilon}\in H^m(\Omega)$ be an eigenfunction associated with $\tilde\mu_{d_{N,m}+1}[\varepsilon]$ normalized by $\int_{\Omega}\tilde\rho_{\varepsilon}u_{\varepsilon}^2dx=1$. Clearly $\|u_{\varepsilon}\|_{L^2(\Omega)}^2\leq\int_{\Omega}\tilde\rho_{\varepsilon}u_{\varepsilon}^2dx= 1$ (see formula \eqref{tilderho}) and since $\int_{\Omega}|D^mu_{\varepsilon}|^2dx\rightarrow 0$ as $\varepsilon\rightarrow 0^+$, we have that the sequence $\left\{u_{\varepsilon}\right\}_{\varepsilon\in]0,\varepsilon_0[}$ is bounded in $H^m(\Omega)$. Then there exists $u_0\in H^m(\Omega)$ such that, up to subsequences, $u_{\varepsilon}\rightharpoonup u_0$ in $H^m(\Omega)$ and $u_{\varepsilon}\rightarrow u_0$ in $H^{m-1}(\Omega)$ as $\varepsilon\rightarrow 0^+$ by the compactness of the embedding $H^m(\Omega)\subset H^{m-1}(\Omega)$.

Since $\lim_{\varepsilon\rightarrow 0^+}\int_{\Omega}|D^m u_{\varepsilon}|^2dx=0$, it is standard  (see e.g., \cite[\S\,5.8]{evans}) to prove that ${\partial^{\alpha} u_0}=0$ in $L^2(\Omega)$ for all $|\alpha|=m$, and hence $u_0=\sum_{|\alpha|\leq m-1}a_{\alpha}x^{\alpha}$ for some constants $\left\{a_{\alpha}\right\}_{|\alpha|\leq m-1}\subset\mathbb R$. This means that $u_0$ is a polymonial of degree at most $m-1$. Moreover, from Lemma \ref{meanlemma}, point $i)$ and $ii)$, it follows that $\partial^{\alpha}u_0(0)=0$ for all $|\alpha|\leq k$, where $k=m-\frac{N}{2}-\frac{1}{2}$ if $N$ is odd, and $k=m-\frac{N}{2}-1$ if $N$ is even. Hence $a_{\alpha}=0$ for all $|\alpha|\leq k$. Then 
\begin{equation}\label{u0-1}
u_0=\sum_{k+1\leq|\alpha|\leq m-1}a_{\alpha}x^{\alpha}.
\end{equation}
Moreover, from Lemma \ref{meanlemma}, points $v)$ and $vi)$ it follows that $\int_{\Omega}u_0x^{\alpha}dx=0$ for all $k+1\leq|\alpha|\leq m-1$ which implies along with \eqref{u0-1} that $u_0\equiv 0$.

Again, from Lemma \ref{meanlemma}, $iv)$ and $v)$, it follows that $1=\int_{\Omega}\tilde\rho_{\varepsilon}u_{\varepsilon}^2dx\rightarrow\int_{\Omega}u_0^2dx$ as $\varepsilon\rightarrow 0^+$, which yields the contradiction. This concludes the proof.

\end{proof}



\section{Weyl-type upper bounds}\label{WE}
In this section we investigate the existence of uniform upper bounds for $\mu_j[\rho]$ which are compatible with the Weyl's law \eqref{weyl}, namely we look for uniform upper bounds of the form
\begin{equation}\label{weylform}
\mu_j[\rho]\leq C_{N,m}\left(\frac{j}{\int_{\Omega}\rho^{\frac{N}{2m}}dx}\right)^{\frac{2m}{N}},
\end{equation}
where the constant $C_{N,m}$ depends only on $m$ and $N$. Actually we will not prove \eqref{weylform}, but a weaker form involving also $\|\rho\|_{L^{\infty}(\Omega)}$ in the case $N<2m$ (see Theorem \ref{weylboundsthm}). We remark that in the case $N=2m$, the bounds \eqref{weylform} hold, in fact this is already contained in Theorem \ref{thm_upperNge4}. Moreover, we shall prove upper bounds involving $\|\rho\|_{L^{\infty}(\Omega)}$ in the case $N>2m$ (see Theorem \ref{sharpbound_weyl}) in which the power of $\|\rho\|_{L^{\infty}(\Omega)}$ turns out to be sharp (see Theorem \ref{counterweyl23}). In particular this implies that bounds of the form \eqref{weylform} do not hold if $N>2m$. We will be left with the open question (see Remark \ref{openquestion}) of the existence of bounds of the form \eqref{weylform} in the case $N<2m$.


\subsection{Upper bounds for \texorpdfstring{$N\leq 2m$}{N<=2m}}\label{sec:41}

In this subsection we prove upper bounds of the form \eqref{weylform} in the case $N\leq 2m$ involving a certain power of $\|\rho\|_{L^{\infty}(\Omega)}$, namely, we prove the following theorem:
\begin{thm}\label{weylboundsthm}
Let $\Omega$ be a bounded domain in $\mathbb R^N$, $N\leq 2m$, such that the embedding $H^m(\Omega)\subset L^2(\Omega)$ is compact. Let $\rho\in\mathcal R$. Then there exists a constant $C_{N,m}>0$ which depends only on $m$ and $N$ such that for all $j\in\mathbb N$ it holds
\begin{equation}\label{weylform2}
\mu_j[\rho]\leq C_{N,m}\left(\frac{|\Omega|\|\rho\|_{L^{\infty}(\Omega)}^{\frac{N}{2m}}}{\int_{\Omega}\rho^{\frac{N}{2m}}dx}\right)^{\frac{2m}{N}-1}\left(\frac{j}{\int_{\Omega}\rho^{\frac{N}{2m}}dx}\right)^{\frac{2m}{N}}.
\end{equation}
\proof
First, we remark that \eqref{weylform2} with $N=2m$ has already been proved in Theorem \ref{thm_upperNge4}. Hence, from now on we let $N<2m$.

The proof is very similar to that of Theorem \ref{thm_upperNge4}. It differs from the choice of the measure $\nu$ in Theorem \ref{gny}. In fact, we exploit Theorem \ref{gny} with $X=\Omega$ endowed with the Euclidean distance and $\nu$ defined by $\nu(E):=\int_{E\cap\Omega}\rho^{\frac{N}{2m}}dx$ for all measurable $E\subset\Omega$.

The hypothesis of Theorem \ref{gny} are clearly satisfied. Then for each index $j\in\mathbb N$ we find $j$ metric annuli $\left\{A_i\right\}_{i=1}^{j}$ such that $2A_i$ are disjoint,
\begin{equation*}
\int_{A_i\cap\Omega}\rho^{\frac{N}{2m}} dx\geq c_N\frac{\int_{\Omega}\rho^{\frac{N}{2m}} dx}{2j},
\end{equation*}
where $c_N>0$ depends only on $N$, and such that
\begin{equation}\label{choice2}
|2A_{i}|\leq\frac{|\Omega|}{j}
\end{equation}
for all $i=1,...,j$.

As in the proof of Theorem \ref{thm_upperNge4}, for all $i=1,...,j$, we define a function $u_{a_{i},r_{i},R_{i}}$ supported on $2A_{i}$ and such that $u_{a_{i},r_{i},R_{i}}\leq 1$ on $2A_{i}$ and $u_{a_{i},r_{i},R_{i}}\equiv 1$ on $A_{i}$ by setting
\begin{equation*}
u_{a_{i},r_{i},R_{i}}(x):=U_{r_{i},R_{i}}(|x-a_{i}|),
\end{equation*}
where $U_{r,R}$ is given by \eqref{radialu}. By construction, the restriction of this function to $\Omega$  belongs to the Sobolev space $H^m(\Omega)$. In order to simplify the notation, we will set $u_{i}(x):=u_{a_{i},r_{i},R_{i}}(x)$.

We have then built a family of $j$ disjointly supported functions belonging to the space $H^m(\Omega)$. We estimate now the Rayleigh quotient $\frac{\int_{\Omega}|D^mu_i|^2dx}{\int_{\Omega}\rho u_i^2dx}$ of $u_i$ for all $i=1,...,j$. For the denominator we have
\begin{multline}\label{den2}
\int_{\Omega}\rho u_i^2dx=\int_{2A_i\cap\Omega}\rho u_i^2dx\geq\int_{A_i\cap\Omega}\rho u_i^2dx=\int_{A_i\cap\Omega}\rho dx\\
\geq \left(\int_{A_i\cap\Omega}\rho^{\frac{N}{2m}}dx\right)^{\frac{2m}{N}}|A_i|^{1-\frac{2m}{N}}\geq c_N\left(\frac{\int_{\Omega}\rho^{\frac{N}{2m}} dx}{2j}\right)^{\frac{2m}{N}}|A_i|^{1-\frac{2m}{N}}.
\end{multline}
This follows from the fact that $u_i\leq 1$ and $u_i\equiv 1$ on $A_i$, from H\"older's inequality and from Theorem \ref{gny}. For the numerator we have
\begin{multline}\label{step12}
\int_{\Omega}|D^mu_i|^2dx=\int_{\Omega\cap 2A_i}|D^mu_i|^2dx\leq\int_{2A_i}|D^mu_i|^2dx\\
\leq C_{N,m} (r_i^{N-2m}+R_i^{N-2m})\leq 2C_{N,m} r_i^{N-2m},
\end{multline}
where $C_{N,m}$ is a constant which depends only on $m$ and $N$. From now on we shall denote by $C_{N,m}$ a positive constant which depends only on $m$ and $N$ and which can be eventually re-defined line by line. Assume now that $A_i=B(a_i,r_i)$ is a ball of center $a_i$ and radius $r_i$. From \eqref{den2} and \eqref{step12} we have that 
\begin{equation}\label{rayball}
\frac{\int_{\Omega}|D^mu_i|^2dx}{\int_{\Omega}\rho u_i^2 dx}\leq C_{N,m}\left(\frac{j}{\int_{\Omega}\rho^{\frac{N}{2m}}dx}\right)^{\frac{2m}{N}},
\end{equation}
Assume now that $A_i$ is a proper annulus (i.e., $0<r_i<R_i$). From Corollary \ref{corollarygny0} it follows that for all $i=1,...,j$
$$
r_i\geq\frac{1}{2}\inf\left\{r\in\mathbb R:V(r)\geq v_j\right\},
$$
where $V(r):=\sup_{x\in\Omega}\int_{B(x,r)}\rho^{\frac{N}{2m}} dx$ and $v_j=c_N\frac{\int_{\Omega}\rho^{\frac{N}{2m}}}{2j}$. As in the proof of Theorem \ref{sharpbounds_mass}, we see that 
\begin{equation}\label{radius_weyl_estimate}
r_i^N\geq\frac{c_N\int_{\Omega}\rho^{\frac{N}{2m}}dx}{2^{N+1}j\omega_N\|\rho\|_{L^{\infty}(\Omega)}^{\frac{N}{2m}}}\geq\frac{c_N\int_{\Omega}\rho^{\frac{N}{2m}}dx|A_i|}{2^{N+1}\omega_N\|\rho\|_{L^{\infty}(\Omega)}^{\frac{N}{2m}}|\Omega|},
\end{equation}
where the second inequality follows from \eqref{choice2}. By combining \eqref{den2}, \eqref{step12} and \eqref{radius_weyl_estimate} we obtain
\begin{equation}\label{rayannulus}
\frac{\int_{\Omega}|D^mu_i|^2dx}{\int_{\Omega}\rho u_i^2 dx}\leq C_{N,m}\left(\frac{|\Omega|\|\rho\|_{L^{\infty}(\Omega)}^{\frac{N}{2m}}}{\int_{\Omega}\rho^{\frac{N}{2m}}dx}\right)^{\frac{2m}{N}-1}\left(\frac{j}{\int_{\Omega}\rho^{\frac{N}{2m}}dx}\right)^{\frac{2m}{N}},
\end{equation}
By combining \eqref{rayball} and \eqref{rayannulus} and by the fact that $\frac{|\Omega|\|\rho\|_{L^{\infty}(\Omega)}^{\frac{N}{2m}}}{\int_{\Omega}\rho^{\frac{N}{2m}}dx}\geq 1$ for all $\rho\in\mathcal R$, we obtain \eqref{weylform2} thanks to \eqref{minimax212} (see also \eqref{minmax1} and \eqref{minmax2}). This concludes the proof.
\end{thm}

\begin{rem}\label{openquestion}
From Theorem \ref{weylboundsthm} it naturally arises the question whether bounds of the form \eqref{weylform} hold in the case $N<2m$. We conjecture an affirmative answer. In fact, in order to produce a family of densities $\left\{\rho_{\varepsilon}\right\}_{\varepsilon\in]0,\varepsilon_0[}$ such that $\mu_j[\rho_{\varepsilon}]\rightarrow +\infty$ as $\varepsilon\rightarrow 0^+$, a necessary condition is that $\rho_{\varepsilon}(x)\rightarrow 0$ for almost every $x\in\Omega$ (otherwise we will find a subset $E\subset\Omega$ of positive measure where $\rho_{\varepsilon}\geq c>0$ for all $\varepsilon\in]0,\varepsilon_0[$ and construct suitable test functions supported in $E$ which can be used to prove upper bounds for all the eigenvalues independent of $\varepsilon$ as is done in Theorem \ref{thm_upperNge4}). Hence, concentration phenomena are the right candidates in order to produce the blow-up of the eigenvalues. We may think to very simple toy models, like concentration around a point or in a neighborhood of the boundary (or in general, in a neighborhood of submanifolds contained in $\Omega$). It is possible, for example, to show that if we concentrate all the mass in a single point, then the eigenvalues remain bounded (one can adapt the same arguments used in the proof of Theorem \ref{counter23} or explicitly construct test functions for the Rayleigh quotient). If we concentrate all the mass in a neighborhood of the boundary, it is possible to prove that $\mu_j[\rho_{\varepsilon}]\rightarrow 0$ as $\varepsilon\rightarrow 0^+$ (see Theorem \ref{lowerbound_counter_weyl} here below). These two types of concentration are somehow the extremal cases of mass concentration around submanifolds contained in $\Omega$.

Moreover, we note that if for a fixed $\rho\in\mathcal R$ and a fixed $j\in\mathbb N$ all the $2j$ annuli given by the decomposition of Theorem \ref{gny} are actually balls, then from \eqref{rayball} we immediately deduce the validity of \eqref{weylform}.
\end{rem}


\subsection{Upper bounds for \texorpdfstring{$N>2m$}{N>2m}}\label{sub:43}
In this subsection we prove upper bounds for the eigenvalues $\mu_j[\rho_{\varepsilon}]$ which involve a suitable power of $\|\rho\|_{L^{\infty}(\Omega)}$. We have the following theorem:

\begin{thm}\label{sharpbound_weyl}
Let $\Omega$ be a bounded domain in $\mathbb R^N$, $N>2m$, such that the embedding $H^m(\Omega)\subset L^2(\Omega)$ is compact. Let $\rho\in\mathcal R$. Then there exists a constant $C_{N,m}>0$ such that for all $j\in\mathbb N$
\begin{equation}\label{sharpbounds_weyl}
\mu_j[\rho]\leq C_{N,m} \left(\frac{|\Omega|\|\rho\|_{L^{\infty}(\Omega)}^{\frac{N}{2m}}}{\int_{\Omega}\rho^{\frac{N}{2m}}dx}\right)^{1-\frac{2m}{N}}\left(\frac{j}{\int_{\Omega}\rho^{\frac{N}{2m}}dx}\right)^{\frac{2m}{N}}.
\end{equation}
\proof
Formula \eqref{sharpbounds_weyl} follows directly from formula \eqref{upperNge4} by observing that for $N>2m$
$$
\int_{\Omega}\rho^{\frac{N}{2m}}dx=\int_{\Omega}\rho\rho^{\frac{N}{2m}-1}dx\leq\|\rho\|_{L^{\infty}(\Omega)}^{\frac{N}{2m}-1}\int_{\Omega}\rho dx,
$$
and hence
\begin{equation}\label{simple}
\left(\int_{\Omega}\rho dx\right)^{-1}\leq\frac{\|\rho\|_{L^{\infty}(\Omega)}^{\frac{N}{2m}-1}}{\int_{\Omega}\rho^{\frac{N}{2m}}dx}.
\end{equation}
By plugging \eqref{simple} into \eqref{upperNge4} and by standard calculus, \eqref{sharpbounds_weyl} immediately follows.
\endproof
\end{thm}


\subsection{Non-existence of Weyl-type upper bounds for \texorpdfstring{$N>2m$}{N>2m} and sharpness of the exponent of \texorpdfstring{$\|\rho\|_{L^{\infty}(\Omega)}$}{Linf2} in \texorpdfstring{\eqref{sharpbounds_weyl}}{4}}\label{sec:42}

In this subsection we prove that for $N>2m$ there exist sequences $\left\{\rho_{\varepsilon}\right\}_{\varepsilon\in]0,\varepsilon_0[}\subset\mathcal R$ such that $\int_{\Omega}\rho_{\varepsilon}^{\frac{N}{2m}}dx\rightarrow|\partial\Omega|$ as $\varepsilon\rightarrow 0^+$, and $\mu_j[\rho_{\varepsilon}]\rightarrow +\infty$ for all $j\geq d_{N,m}+1$ as $\varepsilon\rightarrow 0^+$, and we also provide the rate of divergence to $+\infty$ of the eigenvalues with respect to $\varepsilon$. This means that if $N>2m$ bounds of the form \eqref{weylform} do not hold. This result can be achieved, for example, by concentrating all the mass in a neighborhood of the boundary. Thus, mass densities with fixed $L^{\frac{N}{2m}}$-norm and which concentrate on particular submanifolds may produce blow-up of the eigenvalues if $N>2m$. Moreover the families of densities considered in this subsection will provide the sharpness of the power of $\|\rho\|_{L^{\infty}(\Omega)}$ in \eqref{sharpbounds_weyl}.

Through all this subsection $\Omega$ will be a bounded domain in $\mathbb R^N$ of class $C^{2}$. Let
\begin{equation}\label{strip}
\omega_{\varepsilon}:=\left\{x\in\Omega:{\rm dist}(x,\partial\Omega)<\varepsilon\right\}
\end{equation}
be the $\varepsilon$-tubular neighborhood of $\partial\Omega$. Since $\Omega$ is of class $C^2$ it follows that there exist $\varepsilon_0\in]0,1[$ such that for all $\varepsilon\in]0,\varepsilon_0[$, each point in $\omega_{\varepsilon}$ has a unique nearest point on $\partial\Omega$ (see e.g., \cite{krantz}). For all $\varepsilon\in]0,\varepsilon_0[$ let $\rho_\varepsilon\in\mathcal R$ be defined by
\begin{equation}\label{rhoepsweyl}
\rho_{\varepsilon}:=
\begin{cases}
\varepsilon^{-\frac{2m}{N}},& {\rm in\ } \omega_{\varepsilon},\\
\varepsilon^{2-\frac{2m}{N}},& {\rm in\ } \Omega\setminus\overline\omega_{\varepsilon}.
\end{cases}
\end{equation}
We have the following theorem:

\begin{thm}\label{counterweyl23}
Let $\Omega$ be a bounded domain in $\mathbb R^N$, with $N>2m$, of class $C^2$. Let $\rho_{\varepsilon}\in\mathcal R$ be defined by \eqref{rhoepsweyl} for all $\varepsilon\in]0,\varepsilon_0[$. Then
\begin{enumerate}[i)]
\item $\lim_{\varepsilon\rightarrow 0^+}\int_{\Omega}\rho_{\varepsilon}^{\frac{N}{2m}}dx=|\partial\Omega|$;
\item for all $j\in\mathbb N$, $j\geq d_{N,m}+1$, there exists $c_j>0$ which depends only on $m$, $N$, $\Omega$ and $j$ such that $\lim_{\varepsilon\rightarrow 0^+}\mu_j[\rho_{\varepsilon}]\varepsilon^{1-\frac{2m}{N}}=c_j$.
\end{enumerate}
\end{thm}

From Theorem \ref{counterweyl23} it follows that $\lim_{\varepsilon\rightarrow 0^+}\mu_j[\rho_{\varepsilon}]=+\infty$. Moreover, since $\|\rho_{\varepsilon}\|_{L^{\infty}(\Omega)}=\varepsilon^{-\frac{2m}{N}}$, the power of $\|\rho\|_{L^{\infty}(\Omega)}$ in \eqref{sharpbounds_weyl} is sharp.

In order to prove Theorem \ref{counterweyl23} we will exploit a result on the convergence of the Neumann eigenvalues of the polyharmonic operator to the corresponding Steklov eigenvalues.

The weak formulation of the polyharmonic Steklov eigenvalue problem  reads:
\begin{equation}\label{stweak}
\int_{\Omega}D^mu:D^m\varphi dx=\sigma\int_{\partial\Omega} u \varphi dx\,,\ \ \ \forall\varphi\in H^m(\Omega)
\end{equation}
in the unknowns $u\in H^m(\Omega)$ (the eigenfunction), $\sigma\in\mathbb R$ (the eigenvalue). It is standard to prove that the eigenvalues of \eqref{stweak} are non-negative and of finite multiplicity and are given by
$$
0=\sigma_1=\cdots=\sigma_{d_{N,m}}<\sigma_{d_{N,m}+1}\leq\cdots\leq\sigma_j\leq\cdots\nearrow +\infty.
$$
We refer e.g., to \cite{buosoprovenzano} for a more detailed discussion on the Steklov eigenvalue problem for the biharmonic operator. We have the following theorem:

\begin{thm}\label{Neumann-to-Steklov}
Let $\Omega$ be a bounded domain in $\mathbb R^N$ of class $C^2$. Let $\xi_{\varepsilon}:=\varepsilon+\varepsilon^{-1}\chi_{\omega_{\varepsilon}}$, where $\omega_{\varepsilon}$ is defined by \eqref{strip}. Let $\mu_j[\xi_{\varepsilon}]$ denote the eigenvalues of problem \eqref{nweak} with $\rho=\xi_{\varepsilon}$. Then for all $j\in\mathbb N$
\begin{equation*}
\lim_{\varepsilon\rightarrow 0^+}\mu_j[\xi_{\varepsilon}]=|\partial\Omega|\sigma_j,
\end{equation*}
where $\left\{\sigma_j\right\}_{j\in\mathbb N}$ are the eigenvalues of problem \eqref{stweak}.
\end{thm}
 
We refer to \cite{lambertiprovenzano1,lambertiprovenzanoedinburgh} for the proof of Theorem \ref{Neumann-to-Steklov} in the case of the Laplace operator and to \cite{buosoprovenzano} for the proof of Theorem \ref{Neumann-to-Steklov} in the case of the biharmonic operator, and for more information on the convergence of Neumann eigenvalues to Steklov eigenvalues via mass concentration to the boundary. We remark that the proof of Theorem \ref{Neumann-to-Steklov} for all values of $m\in\mathbb N$ follows exactly the same lines as the proof of the case $m=1$ and $m=2$.

\begin{proof}[Proof of Theorem \ref{counterweyl23}]
We start from point $i)$. It is standard to see that $\int_{\Omega}\rho_{\varepsilon}^{\frac{N}{2m}} dx=\varepsilon^{\frac{2N}{2m}-1}|\Omega\setminus\overline\omega_{\varepsilon}|+\varepsilon^{-1}|\omega_{\varepsilon}|$. The first summand goes to zero as $\varepsilon\rightarrow 0$. For the second summand we note that since $\Omega$ is of class $C^2$, it is standard to prove that $|\omega_{\varepsilon}|=\varepsilon|\partial\Omega|+o(\varepsilon)$ as $\varepsilon\rightarrow 0^+$. This concludes the proof of point $i)$.

We consider now point $ii)$. We note that for all $\varepsilon\in]0,\varepsilon_0[$, $\rho_{\varepsilon}=\varepsilon^{1-\frac{2m}{N}}\tilde\rho_{\varepsilon}$, where
\begin{equation*}
\tilde\rho_{\varepsilon}:=
\begin{cases}
\varepsilon^{-1},& {\rm in\ } \omega_{\varepsilon},\\
\varepsilon,& {\rm in\ } \Omega\setminus\overline\omega_{\varepsilon}.
\end{cases}
\end{equation*}

We note that $\mu\in\mathbb R$ is an eigenvalue of \eqref{nweak} with $\rho=\rho_{\varepsilon}$ if and only if $\tilde\mu:=\varepsilon^{1-\frac{2m}{N}}\mu$ is an eigenvalue of problem \eqref{nweak} with $\rho=\tilde\rho_{\varepsilon}$. Problem \eqref{nweak} with $\rho=\tilde\rho_{\varepsilon}$ admits an increasing sequence of non-negative eigenvalues of finite multiplicity given by
$$
0=\mu_1[\tilde\rho_{\varepsilon}]=\cdots=\mu_{d_{N,m}}[\tilde\rho_{\varepsilon}]<\mu_{d_{N,m}+1}[\tilde\rho_{\varepsilon}]\leq\cdots\leq\mu_j[\tilde\rho_{\varepsilon}]\leq\cdots\nearrow +\infty.
$$

Now from Theorem \ref{Neumann-to-Steklov}, it follows that for all $j\in\mathbb N$,
$$
\lim_{\varepsilon\rightarrow 0}\varepsilon^{1-\frac{2m}{N}}\mu_j[\rho_{\varepsilon}]=\lim_{\varepsilon\rightarrow 0}\mu_j[\tilde\rho_{\varepsilon}]=|\partial\Omega|\sigma_j,
$$
hence, for $j\geq d_{N,m}+1$ $\lim_{\varepsilon\rightarrow 0^+}\varepsilon^{1-\frac{2m}{N}}\mu_j[\rho_{\varepsilon}]=|\partial\Omega|\sigma_j>0$. The proof is now complete.

\end{proof}

\section{Lower bounds}\label{sec:lower}



In this last section we shall discuss the issue of the lower bounds. In many situations (e.g., shape optimization problems) the problem of minimization of the eigenvalues leads to trivial solutions in the case of Neumann boundary conditions. Nevertheless, the eigenvalue problems with density which we have considered in this paper show an interesting behavior with respect to lower bounds, both if we fix the total mass or the $L^{\frac{N}{2m}}$-norm of the density. In the first case, we are able to show that there exist densities which preserve the total mass for which the $j$-th eigenvalue can be made arbitrarily close to zero if $N\geq 2m$ (which is the case when upper bounds with mass constraint exist). This is stated in Theorem \ref{lowerbound_counter_mass}. On the contrary, if $N<2m$, the first positive eigenvalue is uniformly bounded from below by a positive constant which depends only on $m$, $N$ and $\Omega$ divided by the total mass (in this case we recall that upper bounds with mass constraint do not exist). This is stated in Theorem \ref{lowerbound_mass}.

When we choose as a constraint the $L^{\frac{N}{2m}}$-norm of the density, we see that exactly the opposite happens: if $N\leq 2m$ we find densities with prescribed $L^{\frac{N}{2m}}$-norm such that the $j$-th eigenvalue can be made arbitrarily close to zero (in this case we have conjectured the existence of upper bounds of the form \eqref{weylform}), see Theorem \ref{lowerbound_counter_weyl}; if $N>2m$, then the first positive eigenvalue is uniformly bounded from below by a positive constant which depends only on $m$, $N$ and $\Omega$ divided by the $L^{\frac{N}{2m}}$-norm of the density, and in this case we recall that upper bounds with $L^{\frac{N}{2m}}$ constraint do not exist. This is stated in Theorem \ref{lowerbound_weyl}.

We present now the precise statements and the corresponding proofs of such phenomena.

We start with the following theorem concerning lower bounds with mass constraint:

\begin{thm}\label{lowerbound_mass}
Let $\Omega$ be a bounded domain in $\mathbb R^N$, $N<2m$, with Lipschitz boundary. Then there exists a positive constant $C_{m,N,\Omega}$ which depends only on $m$, $N$ and $\Omega$ such that for every $\rho\in\mathcal R$
\begin{equation}\label{poly_sharp_bounds}
\mu_{d_{N,m}+1}[\rho]\geq \frac{C_{m,N,\Omega}}{\int_{\Omega}\rho dx}.
\end{equation}
\proof
We recall from \eqref{minimax212} that
$$
\mu_{d_{N,m}+1}[\rho]=\inf_{\substack{V\leq H^m(\Omega)\\{\rm dim}V=d_{N,m}+1}}\sup_{\substack{0\ne u\in V}}\frac{\int_{\Omega}|D^mu|^2dx}{\int_{\Omega}\rho u^2dx}.
$$
Since $u\in H^m(\Omega)$ and $N<2m$, the standard Sobolev inequality \eqref{sobolev-C} implies that there exists a constant $C$ which depends only on $m$, $N$ and $\Omega$ such that
$$
\|u\|_{C^0(\Omega)}^2\leq C\|u\|_{H^m(\Omega)}^2.
$$
Then for all $u\in H^2(\Omega)$
$$
\frac{\int_{\Omega}|D^mu|^2dx}{\int_{\Omega}\rho u^2dx}\geq\frac{\int_{\Omega}|D^mu|^2dx}{\|u\|_{C^0(\Omega)}^2\int_{\Omega}\rho dx}\geq\frac{1}{C\int_{\Omega}\rho dx}\frac{\int_{\Omega}|D^m u|^2dx}{(\int_{\Omega}|D^m u|^2dx+\int_{\Omega}u^2dx)}.
$$
Hence
\begin{multline*}
\mu_{d_{N,m}+1}[\rho]=\inf_{\substack{V\leq H^m(\Omega)\\{\rm dim}V=d_{N,m}+1}}\sup_{\substack{0\ne u\in V}}\frac{\int_{\Omega}|D^mu|^2dx}{\int_{\Omega}\rho u^2dx}\\
\geq \frac{1}{C\int_{\Omega}\rho dx}\inf_{\substack{V\leq H^m(\Omega)\\{\rm dim}V=d_{N,m}+1}}\sup_{\substack{0\ne u\in V}}\frac{\int_{\Omega}|D^mu|^2dx}{\int_{\Omega}|D^m u|^2dx+\int_{\Omega}u^2dx}\\
\geq\frac{1}{C\int_{\Omega}\rho dx}\frac{\mu_{d_{N,m}+1}[1]}{1+\mu_{d_{N,m}+1}[1]}.
\end{multline*}
Then \eqref{poly_sharp_bounds} holds with $C_{m,N,\Omega}=\frac{\mu_{d_{N,m}+1}[1]}{C(1+\mu_{d_{N,m}+1}[1])}$. This concludes the proof.
\endproof
\end{thm}

Densities which preserve the total mass and produce $j$ arbitrarily small eigenvalues can be given, for example, by concentrating all the mass around $j$ distinct points of $\Omega$. For all $j\in\mathbb N$ let us fix once for all $j$ points $a_1,...,a_j\in\Omega$ and a number $\varepsilon_0\in]0,1[$ such that $B_i^0:=B(a_i,\varepsilon_0)\subset\subset\Omega$ and $B_i^0$ are disjoint. For $\varepsilon\in]0,\varepsilon_0[$, we will write $B_i^{\varepsilon}:=B(a_i,\varepsilon)$. Let $\rho_{\varepsilon,j}\in\mathcal R$ defined by
\begin{equation}\label{rhoeps_lower_mass}
\rho_{\varepsilon,j}:=\varepsilon+\sum_{i=1}^j\varepsilon^{-N}\chi_{B_i^{\varepsilon}}.
\end{equation}
We have the following theorem:
\begin{thm}\label{lowerbound_counter_mass}
Let $\Omega$ be a bounded domain in $\mathbb R^N$, with $N\geq 2m$, such that the embedding $H^m(\Omega)\subset L^2(\Omega)$ is compact. Let $\rho_{\varepsilon,j}\in\mathcal R$ be defined by \eqref{rhoeps_lower_mass} for all $\varepsilon\in]0,\varepsilon_0[$ and $j\in\mathbb N$. Then
\begin{enumerate}[i)]
\item $\lim_{\varepsilon\rightarrow 0^+}\int_{\Omega}\rho_{\varepsilon,j}dx=j\omega_N$;
\item $\mu_j[\rho_{\varepsilon}]\leq C_{N,m}\varepsilon^{N-2m}$ if $N>2m$;
\item $\mu_j[\rho_{\varepsilon}]\leq \frac{C_{m,\Omega,j}}{|\log(\varepsilon)|}$ if $N=2m$,
\end{enumerate}
where $C_{N,m},C_{m,\Omega,j}$ are positive constants which depend only on $m,N$ and $m,\Omega,j$ respectively.
\proof
We start with point $i)$. We have
$$
\int_{\Omega}\rho_{\varepsilon,j}dx=\varepsilon|\Omega|+\sum_{i=1}^j\varepsilon^{-N}|B_i^{\varepsilon}|=\varepsilon|\Omega|+j\omega_N,
$$
which yields the result. 

We prove now $ii)$. Let $N>2m$ and let us fix $j\in\mathbb N$. Let $a_i\in\Omega$, $i=1,...,j$ and $\varepsilon_0\in]0,1[$ be as in the definition of $\rho_{\varepsilon,j}$ in \eqref{rhoeps_lower_mass}. Let $2^{-1}B_i^{\varepsilon}:=B(a_i,\varepsilon/2)$. Associated to each $B_i^{\varepsilon}$ we construct a function $u_i\in H^m(\Omega)$ which is supported in $B_i^{\varepsilon}$ and such that $u_i\equiv 1$ on $2^{-1}B_i^{\varepsilon}$ in the following way:
\begin{equation*}
u_i(x):=
\begin{cases}
U(|x-a_i|), & {\rm if\ } \frac{\varepsilon}{2}\leq |x-a_i| \leq \varepsilon,\\
1, & {\rm if\ } |x-a_i|\leq\frac{\varepsilon}{2},\\
0, & {\rm if\ } |x-a_i|\geq\varepsilon,
\end{cases}
\end{equation*}
where
\begin{equation*}
U(t):=\sum_{i=0}^{2m-1}\frac{\alpha_i}{\varepsilon^i} t^i.
\end{equation*}
The coefficients $\alpha_i$, $i=0,...,2m-1$  are uniquely determined by imposing $U(\varepsilon/2)=1$, $U(\varepsilon)=0$, $U^{(l)}(\varepsilon/2)=U^{(l)}(\varepsilon)=0$ for all $l=1,...,m-1$ and depend only on $m$ (see also \eqref{radialu}, \eqref{system} and \eqref{system2} in the proof of Theorem \ref{thm_upperNge4}).

Now we estimate the Rayleigh quotient $\frac{\int_{\Omega}|D^mu_i|^2dx}{\int_{\Omega}\rho_{\varepsilon,j}u_i^2dx}$ of $u_i$, for all $i=1,..,j$. We start from the numerator. As in the proof of Theorem \ref{thm_upperNge4} (see \eqref{step1} and \eqref{step111}) we have that
\begin{multline}\label{low_num}
\int_{\Omega}|D^m u_i^2|dx=\int_{B_i^{\varepsilon}\setminus 2^{-1}B_i^{\varepsilon}}|D^m u_i|^2dx\leq C_{N,m}|B_i^{\varepsilon}\setminus 2^{-1}B_i^{\varepsilon}|^{1-\frac{2m}{N}}\\
\leq (1-2^{-N})C_{N,m}\varepsilon^{N-2m},
\end{multline}
where $C_{N,m}>0$ depends only on $m$ and $N$ and can be eventually re-defined through the rest of the proof. For the denominator we have
\begin{equation}\label{low_den}
\int_{\Omega}\rho_{\varepsilon,j}u_i^2dx\geq\int_{2^{-1}B_i^{\varepsilon}}\rho_{\varepsilon,j}u_i^2dx=\varepsilon^{-N}|2^{-1}B_i^{\varepsilon}|=2^{-N}\omega_N.
\end{equation}
From \eqref{low_num}, \eqref{low_den} and the min-max principle \eqref{minimax212} and from the fact that $\left\{u_i\right\}_{i=1}^j$ is a set of $j$ disjointly supported functions, it follows that
$$
\mu_j[\rho_{\varepsilon,j}]\leq\max_{u_1,...,u_j}\frac{\int_{\Omega}|D^mu_i|^2dx}{\int_{\Omega}\rho_{\varepsilon,j}u_i^2dx}\leq C_{N,m} \varepsilon^{N-2m}
$$   
(see also \eqref{minmax1}, \eqref{minmax2} in the proof of Theorem \ref{thm_upperNge4}). This concludes the proof of $ii)$.

Consier now $iii)$. Let $N=2m$. Again, let us fix $j\in\mathbb N$. Let $a_i\in\Omega$, $i=1,...,j$ and $\varepsilon_0\in]0,1[$ be as in the definition of \eqref{rhoeps_lower_mass} (we note that admissible values for $\varepsilon_0$ depend on $\Omega$ and $j$). Associated to each $B_i^0$ we construct for all $\varepsilon\in]0,\varepsilon_0[$ a function $u_{\varepsilon,i}\in H^m(\Omega)$ which is supported in $B_i^0$  in the following way:
\begin{equation*}
u_{i,\varepsilon}(x):=
\begin{cases}
U_1(|x-a_i|), & {\rm if\ } \varepsilon\leq |x-a_i| \leq \varepsilon_0,\\
U_2(|x-a_i|), & {\rm if\ } |x-a_i|\leq\varepsilon,\\
0, & {\rm if\ } |x-a_i|\geq\varepsilon_0,
\end{cases}
\end{equation*}
where
\begin{equation*}
U_1(t):=-\log(t)+\log(\varepsilon_0)-\sum_{k=1}^{m-1}\frac{(\varepsilon_0-t)^k}{k\varepsilon_0^k}
\end{equation*}
and
\begin{equation*}
U_2(t):=\alpha(\varepsilon)+\sum_{k=0}^{m-2}\alpha_{k}(\varepsilon)t^{m+k},
\end{equation*}
where the coefficients $\alpha(\varepsilon)$ and $\alpha_k(\varepsilon)$ are uniquely determined by imposing $U_1^{(l)}(\varepsilon)=U_2^{(l)}(\varepsilon)$ for all $0\leq l\leq m-1$. Moreover (possibly choosing a smaller value for $\varepsilon_0$), it is standard to prove that there exist positive constants $c_1$ and $c_2$ which depend only on $m$ and $\varepsilon_0$ (and hence on $m$, $\Omega$ and $j$) such that $c_1|\log(t)|\leq|U_1(t)|\leq c_2|\log(t)|$, $c_1t^{-l}\leq|U_1^{(l)}(t)|\leq c_2t^{-l}$ for all $t\in]\varepsilon,\varepsilon_0[$ and $0\leq l\leq m-1$, and that $c_1|\log(\varepsilon)|\leq|\alpha(\varepsilon)|\leq c_2|\log(\varepsilon)|$ and $c_1\varepsilon^{-m-k}\leq|\alpha_k(\varepsilon)|\leq c_2\varepsilon^{-m-k}$. In particular (possibly re-defining the constants $c_1$ and $c_2$ and choosing a smaller value for $\varepsilon_0$), we have that $c_1|\log(\varepsilon)|\leq|U_2(t)|\leq c_2|\log(\varepsilon)|$ for all $t\in[0,\varepsilon]$.

From the min-max principle \eqref{minimax212} and from the fact that $\left\{u_i\right\}_{i=1}^j$ is a set of $j$ disjointly supported functions, it follows that
\begin{equation}\label{low_minmax}
\mu_j[\rho_{\varepsilon,j}]\leq\max_{u_1,...,u_j}\frac{\int_{\Omega}|D^mu_i|^2dx}{\int_{\Omega}\rho_{\varepsilon,j}u_i^2dx}
\end{equation}
(see also \eqref{minmax1}, \eqref{minmax2} in the proof of Theorem \ref{thm_upperNge4}). It remains then to estimate the Rayleigh quotient of all the function $u_i$. We have for the denominator
\begin{multline}\label{low_den_0}
\int_{\Omega}\rho_{\varepsilon,j}u_i^2dx\geq\varepsilon^{-2m}\int_{B_i^{\varepsilon}}u_i^2dx\geq\varepsilon^{-2m}c_1|\log(\varepsilon)|^2|B_i^{\varepsilon}|\geq C_{m,\Omega,j}|\log(\varepsilon)|^2.
\end{multline}
From now on $C_{m,\Omega,j}$ will denote a positive constat which depends only on $m,\Omega$ and $j$. For the numerator, we have, since the functions $u_i$ are radial with respect to $a_i$ (see also \eqref{hessianintegral})
\begin{multline}\label{low_num_0}
\int_{\Omega}|D^m u_i|^2dx=\int_{B_i^0}|D^m u_i|^2dx=\int_{B_i^0\setminus B_i^{\varepsilon}}|D^m u_i|^2dx+\int_{B_i^{\varepsilon}}|D^m u_i|^2dx\\
\leq C_{m,\Omega,j}\sum_{k=1}^m\int_{B_i^0\setminus B_i^{\varepsilon}}\frac{(U_1^{(k)}(|x-a_i|))^2}{|x-a_i|^{2(m-k)}}dx+ C_{m,\Omega,j}\sum_{k=1}^m\int_{B_i^{\varepsilon}}\frac{(U_2^{(k)}(|x-a_i|))^2}{|x-a_i|^{2(m-k)}}dx.
\end{multline}
Since $|U_1^{(k)}(t)|^2\leq c_2^2t^{-2k}$ and $N=2m$, we have
\begin{equation}\label{low_num_1}
\int_{B_i^0\setminus B_i^{\varepsilon}}\frac{(U_1^{(k)}(|x-a_i|))^2}{|x-a_i|^{2(m-k)}}dx\leq N C_{m,\Omega,j}\int_{\varepsilon}^{\varepsilon_0}t^{-2m+N-1}dt\leq C_{m,\Omega,j}|\log(\varepsilon)|.
\end{equation}
Moreover, $|U_2^{(k)}(t)|^2\leq (m-1)c_2^2\sum_{i=0}^{m-2}\varepsilon^{-2m-2i}t^{2m+2i-2k}$, hence
\begin{equation}\label{low_num_2}
\sum_{k=1}^m\int_{B_i^{\varepsilon}}\frac{(U_2^{(k)}(|x-a_i|))^2}{|x-a_i|^{2(m-k)}}dx\leq C_{m,\Omega,j}\sum_{i=0}^{m-2}\varepsilon^{-2m-2i}\int_0^{\varepsilon}t^{2i+N-1}\leq C_{m,\Omega,j}.
\end{equation}
From \eqref{low_num_0}, \eqref{low_num_1} and \eqref{low_num_2} we have that
\begin{equation}\label{low_num_3}
\int_{\Omega}|D^m u_i|^2dx\leq C_{m,\Omega,j}|\log(\varepsilon)|.
\end{equation}
By combining \eqref{low_den_0} and \eqref{low_num_3}, from \eqref{low_minmax} we deduce that $\mu_j[\rho_{\varepsilon,j}]\leq \frac{C_{m,\Omega,j}}{|\log(\varepsilon)|}$. This concludes the proof for the case $N=2m$ and of the theorem.
\endproof
\end{thm}

We consider now lower bounds with the non-linear constraint $\int_{\Omega}\rho^{\frac{N}{2m}}dx={\rm const.}$ and $N>2m$. We have the following theorem:

\begin{thm}\label{lowerbound_weyl}
Let $\Omega$ be a bounded domain in $\mathbb R^N$, $N>2m$, with Lipschitz boundary. Then there exists a positive constant $C_{m,N,\Omega}$ which depends only on $m$, $N$ and $\Omega$ such that for every $\rho\in\mathcal R$
\begin{equation}\label{poly_sharp_bounds_weyl}
\mu_{d_{N,m}+1}[\rho]\geq \frac{C_{m,N,\Omega}}{\left(\int_{\Omega}\rho^{\frac{N}{2m}} dx\right)^{\frac{2m}{N}}}.
\end{equation}
\proof
We recall from \eqref{minimax212} that
$$
\mu_{d_{N,m}+1}[\rho]=\inf_{\substack{V\leq H^m(\Omega)\\{\rm dim}V=d_{N,m}+1}}\sup_{\substack{0\ne u\in V}}\frac{\int_{\Omega}|D^mu|^2dx}{\int_{\Omega}\rho u^2dx}.
$$
Since $u\in H^m(\Omega)$ and $N>2m$, the standard Sobolev inequality \eqref{sobolev0} implies that there exists a constant $C$ which depends only on $m$, $N$ and $\Omega$ such that
$$
\|u\|_{L^{\frac{2N}{N-2m}}(\Omega)}^2\leq C\|u\|_{H^m(\Omega)}^2.
$$
Then for all $u\in H^2(\Omega)$
\begin{multline*}
\frac{\int_{\Omega}|D^mu|^2dx}{\int_{\Omega}\rho u^2dx}\geq\frac{\int_{\Omega}|D^m u|^2dx}{\left(\int_{\Omega}\rho^{\frac{N}{2m}}dx\right)^{\frac{2m}{N}}\left(\int_{\Omega}u^{\frac{2N}{N-2m}}dx\right)^{\frac{N-2m}{N}}}\\
\geq\frac{1}{C\|\rho\|_{L^{\frac{N}{2m}}(\Omega)}}\frac{\int_{\Omega}|D^m u|^2dx}{\int_{\Omega}|D^m u|^2dx+\int_{\Omega}u^2dx},
\end{multline*}
where we have used a H\"older inequality in the first line. Hence
\begin{multline*}
\mu_{d_{N,m}+1}[\rho]=\inf_{\substack{V\leq H^m(\Omega)\\{\rm dim}V=d_{N,m}+1}}\sup_{\substack{0\ne u\in V}}\frac{\int_{\Omega}|D^mu|^2dx}{\int_{\Omega}\rho u^2dx}\\
\geq \frac{1}{C\|\rho\|_{L^{\frac{N}{2m}}(\Omega)}}\inf_{\substack{V\leq H^m(\Omega)\\{\rm dim}V=d_{N,m}+1}}\sup_{\substack{0\ne u\in V}}\frac{\int_{\Omega}|D^mu|^2dx}{\int_{\Omega}|D^m u|^2dx+\int_{\Omega}u^2dx}\\
\geq\frac{1}{C\|\rho\|_{L^{\frac{N}{2m}}(\Omega)}}\frac{\mu_{d_{N,m}+1}[1]}{1+\mu_{d_{N,m}+1}[1]}.
\end{multline*}
Hence formula \eqref{poly_sharp_bounds_weyl} holds with $C_{m,N,\Omega}=\frac{\mu_{d_{N,m}+1}[1]}{C(1+\mu_{d_{N,m}+1}[1])}$. This ends the proof.
\endproof
\end{thm}

Densities with prescribed $L^{\frac{N}{2m}}$-norm and which made the $j$-th eigenvalue arbitrarily small in dimension $N<2m$ are, for example, densities which explode in a $\varepsilon$-tubular neighborhood of the boundary of $\Omega$. This is contained in the following theorem.
\begin{thm}\label{lowerbound_counter_weyl}
Let $\Omega$ be a bounded domain in $\mathbb R^N$ with $N<2m$, of class $C^2$. Let $\rho_{\varepsilon}\in\mathcal R$ be defined by \eqref{rhoepsweyl} for all $\varepsilon\in]0,\varepsilon_0[$ and $j\in\mathbb N$. Then
\begin{enumerate}[i)]
\item $\lim_{\varepsilon\rightarrow 0^+}\int_{\Omega}\rho_{\varepsilon}^{\frac{N}{2m}}dx=|\partial\Omega|$;
\item for all $j\in\mathbb N$, $j\geq d_{N,m}+1$, there exists $c_j>0$ which depends only on $m$, $N$, $\Omega$ and $j$ such that $\lim_{\varepsilon\rightarrow 0^+}\mu_j[\rho_{\varepsilon}]\varepsilon^{1-\frac{2m}{N}}=c_j$.
\end{enumerate}
\proof
The proof is the same as that of Theorem \ref{counterweyl23} and is accordingly omitted.
\endproof
\end{thm}

From Theorem \ref{lowerbound_counter_weyl} it follows that $\lim_{\varepsilon\rightarrow 0^+}\mu_j[\rho_{\varepsilon}]=0$ for all $j\in\mathbb N$ if $N<2m$. We remark that also in the case $N=2m$ we find densities which make the $j$-th eigenvalue arbitrarily small, in fact this is stated by point $iii)$ of Theorem \ref{lowerbound_counter_mass}.





\appendix
\addcontentsline{toc}{section}{Appendices}
\section*{Appendices}

\section{Eigenvalues of polyharmonic operators}

In this section we shall present some basics of spectral theory for the polyharmonic operators. In particular, we will discuss Neumann boundary conditions, mainly for the Laplace and the biharmonic operator. Then we will characterize the spectrum of the polyharmonic operators subjet to Neumann boundary conditions by exploiting classical tools of spectral theory for compact self-adjoint operators. We refer to \cite{egorov,laproeurasian} and to the references therein for a discussion on eigenvalue problems for general elliptic operators of order $2m$ with density subject to homogeneous boundary conditions. 

\subsection{Neumann boundary conditions}\label{app:neumann}
Neumann boundary conditions are usually called `natural' boundary conditions. This is well understood for the Laplace operator. In fact, assume that $u$ is a classical solution of \eqref{nclassic1}. If we multiply the equation $-\Delta u=\mu\rho u$ by a test function $\varphi\in C^{\infty}(\Omega)$ and integrate both sides of the resulting identity over $\Omega$, thanks to Green's formula we obtain:
$$
\int_{\Omega}\nabla u\cdot\nabla\varphi dx=\mu\int_{\Omega}\rho u\varphi dx+\int_{\partial\Omega}\frac{\partial u}{\partial\nu}\varphi d\sigma=\mu\int_{\Omega}\rho u\varphi dx.
$$
Hence \eqref{nweak} with $m=1$ holds for all $\varphi\in C^{\infty}(\Omega)$ when $u$ is a solution of \eqref{nclassic}. We can relax our hypothesis on $u$ and just require that $u\in H^1(\Omega)$ and that \eqref{nweak} holds for all $\varphi\in H^1(\Omega)$. Hence \eqref{nweak} is the weak formulation of the Neumann eigenvalue problem for the Laplace operator. We note that the boundary condition in \eqref{nclassic} arises naturally and is not imposed a priori with the choice of a subspace of $H^1(\Omega)$ in the weak formulation (as in the case of $H^1_0(\Omega)$ for Dirichlet conditions): if a weak solution of \eqref{nweak} for $m=1$ exists and is sufficiently smooth, then it solves $-\Delta u=\lambda\rho u$ in the classical sense and satisfies the Neumann boundary condition $\frac{\partial u}{\partial\nu}=0$.

Let us consider now more in detail the case of the biharmonic operator. Assume that $u$ is a classical solution of problem \eqref{nclassic2}. We multiply the equation $\Delta^2u=\mu\rho u$ by a test function $\varphi\in C^{\infty}(\Omega)$ and apply the biharmonic Green's formula (see \cite[Lemma\,8.56]{arrietalamberti1}). We obtain:
 \begin{multline*}
 \int_\Omega\Delta^2 u\varphi dx=\int_{\Omega}D^2u:D^2\varphi dx\\+\int_{\partial\Omega}\left({\rm div}_{\partial\Omega}\left(D^2u.\nu\right)_{\partial\Omega}+\frac{\partial\Delta u}{\partial\nu}\right)\varphi d\sigma-\int_{\partial\Omega}\frac{\partial^2 u}{\partial\nu^2}\frac{\partial\varphi}{\partial\nu} d\sigma\\
=\int_{\Omega}D^2u:D^2\varphi dx=\mu\int_{\Omega}\rho u\varphi dx,
\end{multline*}
where $(D^2u.\nu)_{\partial\Omega}$ denotes the tangential component of  $D^2u.\nu$. Hence \eqref{nweak} with $m=2$ holds for all $\varphi\in C^{\infty}(\Omega)$ when $u$ is a solution of problem \eqref{nclassic2} (we remark that if $\frac{\partial^2u}{\partial\nu^2}=0$ then $(D^2u.\nu)_{\partial\Omega}=(D^2u.\nu)$). We can relax our hypothesis on $u$ and just require that $u\in H^2(\Omega)$ and that \eqref{nweak} holds for all $\varphi\in H^2(\Omega)$. This is exactly the weak formulation of the Neumann eigenvalue problem for the biharmonic operator. We note again that the two boundary conditions in \eqref{nclassic2} arise naturally and are not imposed a priori: if a weak solution of \eqref{nweak} exists and is sufficiently smooth, then it satisfies the two Neumann boundary conditions. We also remark that if $\Omega$ is sufficiently regular, e.g., if it is of class $C^k$ with $k>4+\frac{N}{2}$ and $\rho$ is continuous,  then a weak solution of \eqref{nweak} with $m=2$ is actually a classical solution of \eqref{nclassic2} (see \cite[\S\,2]{gazzola}). The choice of the whole space $H^2(\Omega)$ in the weak formulation \eqref{nweak} contains the information on the boundary conditions in \eqref{nclassic}. 

It is natural then to consider problem \eqref{nweak} for any $m\in\mathbb N$ as the weak formulation of an eigenvalue problem for the polyharmonic operator with Neumann boundary conditions. In the case of a generic value of $m$ it is much more difficult to write explicitly the boundary operators $\mathcal N_0,...,\mathcal N_{m-1}$ (this is already extremely involved for $m=3$).  If moreover $\Omega$ is sufficiently regular and $\rho$ is continuous, then weak solutions of \eqref{nweak} are actually classical solution of \eqref{nclassic}, and the $m$ bounday conditions are uniquely determined and arise naturally from the choice of the whole space $H^m(\Omega)$ (see \cite{gazzola} for further discussions on higher order elliptic operators and eigenvalue problems).

\subsection{Characterization of the spectrum}

The aim of this subsection is to prove that problem \eqref{nweak} admits an increasing sequence of non-negative eigenvalues of finite multiplicity diverging to $+\infty$, and to provide some additional information on the spectrum. To do so, we will reduce problem (\ref{nweak}) to an eigenvalue problem for a compact self-adjoint operator on a Hilbert space. 

We define first the following (equivalent) problem: find $u\in H^m(\Omega)$ and $\Lambda\in\mathbb R$ such that
\begin{equation}\label{weakneumann2}
\int_{\Omega}D^mu:D^m\varphi +\rho u\varphi dx=\Lambda\int_{\Omega}\rho u\varphi dx\,,\ \ \ \forall \varphi\in H^m(\Omega).
\end{equation}
Clearly the eigenfuctions of \eqref{weakneumann2} coincide with the eigenfunctions of \eqref{nweak}, while all the eigenvalues $\mu$ of \eqref{nweak} are given by $\mu=\Lambda-1$, where $\Lambda$ is an eigenvalue of \eqref{weakneumann2}.

We consider the operator $(-\Delta)^m+\rho I_d$ as a map from $H^m(\Omega)$ to its dual $H^m(\Omega)'$ defined by
\begin{equation*}
((-\Delta)^m+\rho I_d)[u][\varphi]:=\int_\Omega D^m u: D^m\varphi +\rho u\varphi dx,\ \ \ \mathcal8\varphi\in H^m(\Omega).
\end{equation*}
The operator $(-\Delta)^m+\rho I_d$ is a continuous isomorphism between $H^m(\Omega)$ and $H^m(\Omega)'$. In fact it follows immediately that there exist $C_1,C_2>0$ such that
\begin{equation}\label{scalarp}
C_1\|u\|^2_{H^m(\Omega)}\leq((-\Delta)^m+\rho I_d)[u][u]\leq C_2 \|u\|^2_{H^m(\Omega)},\ \forall\,u\in H^m(\Omega).
\end{equation}

Next we denote by $i$ the canonical embedding of $H^m(\Omega)$ into $L^2(\Omega)$ and by $J_\rho$ the embedding of $L^2(\Omega)$ into $H^m(\Omega)'$, defined by
$$
J_\rho[u][\varphi]:=\int_\Omega\rho u\varphi dx\ \ \ \mathcal8 u\in L^2(\Omega),\varphi\in H^m(\Omega).
$$
Let $T_\rho$ be the operator from $H^m(\Omega)$ to itself defined by $T_\rho:=((-\Delta)^m+\rho I_d)^{(-1)}\circ J_\rho\circ i$. Problem (\ref{weakneumann2}) is then equivalent to
\begin{equation*}
T_\rho u =\Lambda^{-1}u,
\end{equation*}
in the unknows $u\in H^m(\Omega)$, $\Lambda\in\rea$. We now consider the space $H^m(\Omega)$ endowed with the bilinear form
\begin{equation}\label{bilinear884}
\langle u,v\rangle_{\rho}=\int_\Omega D^mu:D^mv+\rho u v dx,\ \ \ \mathcal8 u,v\in H^m(\Omega).
\end{equation}
From \eqref{scalarp} it follows that \eqref{bilinear884} is a scalar product on $H^m(\Omega)$ whose induced norm is equivalent to the standard one. We denote by ${H^m_{\rho}}(\Omega)$ the space $H^m(\Omega)$ endowed with the scalar product defined by (\ref{bilinear884}). Then we can state the following theorem:

\begin{thm}\label{lemma_neumann_883}
Let $\Omega$ be a bounded domain in $\rea^N$ such that the embedding $H^m(\Omega)\subset L^2(\Omega)$ is compact. Let $\rho\in\mathcal R$. The operator $T_\rho:=((-\Delta)^m+\rho I_d)^{(-1)}\circ J_\rho\circ i$ is a compact self-adjoint operator in ${H^m_{\rho}}(\Omega)$, whose eigenvalues coincide with the reciprocals of the eigenvalues of problem (\ref{weakneumann2}) for all $j\in\mathbb N$.
\end{thm}
The proof of Theorem \ref{lemma_neumann_883} is standard, hence we omit it (see e.g., \cite[\S\,IX]{brezis}). As a consequence of Theorem \ref{lemma_neumann_883} we have the following:

\begin{thm}\label{teorema_neumann_883}
Let $\Omega$ be a bounded domain in $\rea^N$ such that the embedding $H^m(\Omega)\subset L^2(\Omega)$ is compact. Let $\rho\in\mathcal R$. Then the set of the eigenvalues of (\ref{nweak}) is contained in $[0,+\infty[$ and consists of the image of a sequence increasing to $+\infty$. The eigenvalue $\mu=0$ has multiplicity $d_{N,m}=\binom{N+m-1}{N}$ and the eigenfunctions corresponding to the eigenvalue $\mu=0$ are the polynomials of degree at most $m-1$ in $\mathbb R^N$. Each eigenvalue has finite multiplicity. Moreover the space $H^m_{\rho}(\Omega)$ has a Hilbert basis of eigenfunctions of problem \eqref{nweak}.
\proof
We note that $ker(T_{\rho})=\left\{0\right\}$, hence by standard spectral theory it follows that the eigenvalues of $T_{\rho}$ are positive  and bounded and form an infinite sequence $\left\{\lambda_j\right\}_{j\in\mathbb N}$ converging to zero. Moreover to each eigenvalue $\lambda_j$ is possible to associate an eigenfunction $u_j$ such that $\left\{u_j\right\}_{j\in\mathbb N}$ is a orthonormal basis of $H^m_{\rho}(\Omega)$. 

From Theorem \ref{lemma_neumann_883} it follows that the eigenvalues of \eqref{weakneumann2} form a sequence of real numbers increasing to $+\infty$ which is given by $\left\{\Lambda_j=\lambda_j^{-1}\right\}_{j\in\mathbb N}$ and that the space $H^m_{\rho}(\Omega)$ has a Hilbert basis of eigenfunctions of \eqref{weakneumann2}. The eigenvalues $\mu_j$ of \eqref{nweak} are given by $\mu_j=\Lambda_j-1$ for all $j\in\mathbb N$, where $\left\{\Lambda_j\right\}_{j\in\mathbb N}$ are the eigenvalues of \eqref{weakneumann2} and the eigenfunctions associated with $\Lambda_j$ coincide with the eigenfunctions associated with $\mu_j=\Lambda_j-1$. Moreover, given an eigenvalue $\mu$ of \eqref{nweak} and a corresponding eigenfunction $u$, we have that
$$
\int_{\Omega}|D^m u|^2dx=\mu\int_{\Omega}\rho u^2 dx,
$$
thus $\mu\in[0,+\infty[$. Finally, if $\mu=0$, then $\int_{\Omega}|D^m u|^2dx=0$, thus $u$ is a polynomial of degree at most $m-1$ in $\mathbb R^N$. The eigenspace associated with the eigenvalue $\mu=0$ has dimension $\binom{N+m-1}{N}$ and coincides with the space of the polynomials of degree at most $m-1$ in $\mathbb R^N$. This concludes the proof.
\end{thm}

\section{A few useful functional inequalities}\label{app:func}

In this section we will prove some useful functional inequalities which are crucial in the proof of the results of Subsection \ref{sub:32}, in particular  of Theorem \ref{counter23}. Since we think that they are interesting on their own, we shall provide here all the details of the proofs. Through this section $\Omega$ will be a bounded domain in $\mathbb R^N$ with Lipschitz boundary. We start this section by recalling the standard Sobolev embeddings.

\begin{thm}\label{sobolev_embedding}
Let $\Omega$ be a bounded domain with Lipschitz boundary. Let $m\in\mathbb N$ and assume that $u\in H^m(\Omega)$.
\begin{enumerate}[i)]
\item If $N<2m$ then $u\in C^{m-\left[\frac{N}{2}\right]-1,\gamma}(\Omega)$, where 
\begin{equation*}
\gamma=
\begin{cases}
\left[\frac{N}{2}\right]+1-\frac{N}{2}, & {\rm if\ } \frac{N}{2}\notin\mathbb N\\
{\rm any\ number\ in\ } ]0,1[,& {\rm if\ } \frac{N}{2}\in\mathbb N.
\end{cases}
\end{equation*}
Moreover there exists a positive constant $C$ which depends only on $m,N$ and $\Omega$ such that
\begin{equation}\label{sobolev-C}
\|u\|_{C^{m-\left[\frac{N}{2}\right]-1,\gamma}(\Omega)}\leq C\|u\|_{H^m(\Omega)}.
\end{equation}
\item If $N>2m$ then $u\in L^{\frac{2N}{N-2m}}(\Omega)$ and
\begin{equation}\label{sobolev0}
 \|u\|_{L^{\frac{2N}{N-2m}}(\Omega)}\leq C\|u\|_{H^m(\Omega)}
\end{equation}
the constant $C$ depending only on $m,N$ and $\Omega$.
\item  If $N=2m$ there exist constants $C_1, C_2>0$ which depend only on $m$ and $\Omega$ such that
\begin{equation}\label{mosertrudinger}
\int_{\Omega}e^{C_1\left(\frac{u(x)}{\|u\|_{H^{m}(\Omega)}}\right)^2}dx\leq C_2.
\end{equation}
\end{enumerate}
\end{thm}
We refer to \cite[\S\,4.6-4.7]{burenkov} and to \cite[\S\,5.6.3]{evans} for the proof ot points $i)$ and $ii)$ of Theorem \ref{sobolev_embedding}. We refer to \cite[Theorem\,1.1]{cianchimosertrudinger} for the proof of \eqref{mosertrudinger} (see also \cite[\S\,4.7]{burenkov}).

From Theorem \ref{sobolev_embedding} it follows that if $N<2m$ then  a function $u\in H^m(\Omega)$ is (equivalent to) a function of class $C^{m-\left[\frac{N}{2}\right]-1}(\Omega)$. In particular if $N$ is odd, we can write $N=2m-2k-1$ for some $k\in\left\{0,...,m-1\right\}$ and a function $u\in H^m(\Omega)$ is (equivalent to) a function of class $C^{k,\frac{1}{2}}(\Omega)$. If $N$ is even, we can write $N=2m-2k-2$ for some $k\in\left\{0,...,m-2\right\}$ and a function $u\in H^m(\Omega)$ is (equivalent to) a function of class $C^{k,\gamma}(\Omega)$ for any $\gamma\in]0,1[$.

Assume now that a function $u\in H^{m}(\Omega)$ has all its partial derivatives up to the $k$-th order vanishing at a point $x_0\in\Omega$. Then the integral of $u^2$ over $B(x_0,\varepsilon)$ (where $\varepsilon>0$ is such that $B(x_0,\varepsilon)\subset\subset\Omega$) can be controlled by $\varepsilon^{2m}\|u\|^2_{H^m(\Omega)}$ if $N<2m$ is odd, and by $\varepsilon^{2m}(1+|\log(\varepsilon)|)\|u\|^2_{H^m(\Omega)}$ if $N<2m$ is even. This is proved in the following lemma, where without loss of generality we set $x_0=0$.

\begin{lemma}\label{pre2}
Let $\Omega$ be a bounded domain in $\mathbb R^N$, $N<2m$, with Lipschitz boundary. Assume that $0\in\Omega$ and let $\varepsilon>0$ be such that $B(0,\varepsilon)\subset\subset\Omega$. Let $u\in H^m(\Omega)$. Then there exists a positive constant $C$ which depends only on $m,k$ and $\Omega$ such that
\begin{enumerate}[i)]
\item $\int_{B(0,\varepsilon)}{\left|u(x)-\sum_{|\alpha|\leq k}\frac{\partial^{\alpha}u(0)}{\alpha!}x^{\alpha}\right|^2dx}\leq C\varepsilon^{2m}\|u\|_{H^m(\Omega)}^2$ if $N=2m-2k-1$, $\forall\,k=0,...,m-1$;
\item $\int_{B(0,\varepsilon)}{\left|u(x)-\sum_{|\alpha|\leq k}\frac{\partial^{\alpha}u(0)}{\alpha!}x^{\alpha}\right|^2dx}\leq C\varepsilon^{2m}(1+|\log(\varepsilon)|)\|u\|_{H^m(\Omega)}^2$ if $N=2m-2k-2$, $\forall\,k=0,...,m-2$.
\end{enumerate}
\proof
We start by proving $i)$. Let $N=2m-2k-1$ for some $k\in\left\{0,...,m-1\right\}$. Actually, we will prove $i)$ for a function $u\in C^{k+1}(\Omega)\cap H^m(\Omega)$. The result for a function $u\in H^m(\Omega)$ will follow from standard approximation of functions in the space $H^m(\Omega)$ by smooth functions (see \cite[\S\,2.3]{burenkov} and \cite[\S\,5.3]{evans}). Let then $u\in C^{k+1}(\Omega)\cap H^m(\Omega)$. Through the rest of the proof we shall denote by $C$ a positive constant which depends only on $m,k$ and $\Omega$ and which we can eventually re-define line by line. From the standard Sobolev embedding theorem, it follows that $u\in C^{k+1}(\Omega)\cap C^{k,\frac{1}{2}}(\Omega)$. From Taylor's Theorem it follows that
\begin{equation}\label{taylor0}
u(x)-\sum_{|\alpha|\leq k}\frac{\partial^{\alpha}u(0)}{\alpha !}x^{\alpha}=\sum_{|\beta|=k+1}\frac{|\beta|}{\beta !}\int_0^1(1-t)^k\partial^{\beta}u(tx)dt x^{\beta}.
\end{equation}
We consider now the absolute value of the expression in the right-hand side of \eqref{taylor0} and integrate over $B(0,\varepsilon)$ each integral which appears in the sum. We have
\begin{multline}\label{est-tay}
\int_{B(0,\varepsilon)}\left|\int_0^1(1-t)^k\partial^{\beta}u(tx)x^{\beta}dt\right|dx\leq\int_0^1(1-t)^k\int_{B(0,\varepsilon)}|\partial^{\beta}u(tx)||x|^{k+1}dxdt\\
=\int_0^1t^{-2m+k}(1-t)^k\int_{B(0,t\varepsilon)}|\partial^{\beta}u(y)||y|^{k+1}dydt\\
\leq \int_0^1t^{-2m+k}(1-t)^k \left(\int_{B(0,t\varepsilon)}|\partial^{\beta}u(y)|^{2 (2m-2k-1)}dy\right)^{\frac{1}{2(2m-2k-1)}}\\
\cdot\left(\int_{B(0,t\varepsilon)}|y|^{\frac{2(2m-2k-1)(k+1)}{2(2m-2k-1)-1}}dy\right)^{\frac{2(2m-2k-1)-1}{2(2m-2k-1)}}dt\\
\leq C\int_0^1t^{-2m+k}(1-t)^k \|\partial^{\beta}u\|_{L^{2(2m-2k-1)}(\Omega)}\varepsilon^{2m-k-\frac{1}{2}}t^{2m-k-\frac{1}{2}}dt\\
=C\varepsilon^{2m-k-\frac{1}{2}}\|\partial^{\beta}u\|_{L^{2(2m-2k-1)}(\Omega)}\leq C\varepsilon^{2m-k-\frac{1}{2}}\|\partial^{\beta}u\|_{H^{m-k-1}(\Omega)} \\
\leq C\varepsilon^{2m-k-\frac{1}{2}}\|u\|_{H^{m}(\Omega)},
\end{multline}
where in the last line we have used the Sobolev inequality \eqref{sobolev0} for functions in $H^{m-k-1}(\Omega)$ with $\Omega\in\mathbb R^{2m-2k-1}$.

Next, we estimate the quantity 
$$
\left\|u(x)-\sum_{|\alpha|\leq k}\frac{\partial^{\alpha}u(0)}{\alpha!}x^{\alpha}\right\|_{C^{0}(B(0,\varepsilon))}:=\max_{x\in B(0,\varepsilon)}\left|u(x)-\sum_{|\alpha|\leq k}\frac{\partial^{\alpha}u(0)}{\alpha!}x^{\alpha}\right|
$$
for a function $u\in C^{k,\frac{1}{2}}(\Omega)$. First we note that there exits $t\in]0,1[$ such that
\begin{equation*}
u(x)-\sum_{|\alpha|\leq {k-1}}\frac{\partial^{\alpha}u(0)}{\alpha!}x^{\alpha}=\sum_{|\alpha|=k}\partial^{\alpha}u(tx)\frac{x^{\alpha}}{\alpha!},
\end{equation*}
then
\begin{multline*}
\left|u(x)-\sum_{|\alpha|\leq k}\frac{\partial^{\alpha}u(0)}{\alpha!}x^{\alpha}\right|=\left|\sum_{|\alpha|=k}\frac{x^{\alpha}}{\alpha !}(\partial^{\alpha}u(0)-\partial^{\alpha}u(tx)\right|\\
\leq |x|^k\sum_{|\alpha|=k}\frac{1}{\alpha !}|\partial^{\alpha}u(0)-\partial^{\alpha}u(tx)|,
\end{multline*}
which implies
\begin{equation}\label{unif-1}
\left\|u(x)-\sum_{|\alpha|\leq k}\frac{\partial^{\alpha}u(0)}{\alpha!}x^{\alpha}\right\|_{C^{0}(B(0,\varepsilon))}\leq C\varepsilon^{k+\frac{1}{2}}\|u\|_{H^m(\Omega)},
\end{equation}
where the constant $C$ depends only on $m,k$ and $\Omega$ (see also \eqref{sobolev-C}).

Consider now \eqref{taylor0}. We take the squares of both sides and integrate over $B(0,\varepsilon)$. We have
\begin{multline}\label{final-1}
\int_{B(0,\varepsilon)}\left|u(x)-\sum_{|\alpha|\leq k}\frac{\partial^{\alpha}u(0)}{\alpha !} x^{\alpha}\right|^2dx\\
\leq \left\|u(x)-\sum_{|\alpha|\leq k}\frac{\partial^{\alpha}u(0)}{\alpha!}x^{\alpha}\right\|_{C^{0}(B(0,\varepsilon))} \int_{B(0,\varepsilon)}\left|u(x)-\sum_{|\alpha|\leq k}\frac{\partial^{\alpha}u(0)}{\alpha !} x^{\alpha}\right|dx\\
\leq C \varepsilon^{2m}\|u\|_{H^m(\Omega)}^2,
\end{multline}
where the last inequality follows from \eqref{est-tay} and \eqref{unif-1} and the constant $C$ depends only on $m,k$ and $\Omega$. Since inequality \eqref{final-1} holds for all $u\in C^{k+1}(\Omega)\cap H^m(\Omega)$, by standard approximation of $H^m(\Omega)$ functions by smooth functions, we conclude that it holds for all $u\in H^m(\Omega)$. This proves $i)$.

Consider now $ii)$. Let $N=2m-2k-2$ for some $k\in\left\{0,...,m-2\right\}$. Again, we shall prove $ii)$ for a function $u\in C^{k+1}(\Omega)\cap H^m(\Omega)$. The result for a function $u\in H^m(\Omega)$ will follows from standard approximation of functions in the space $H^m(\Omega)$ by smooth functions.

We prove first the following inequality:
\begin{equation}\label{odd-log}
\|f\|_{L^2(B(0,\varepsilon))}^2\leq C\varepsilon^{2m-2k-2}(1+|\log(\varepsilon)|)\|f\|_{H^{m-k-1}(\Omega)}^2,
\end{equation}
for all $f\in H^{m-k-1}(\Omega)$ (the constant $C>0$ depending only on $m,k$ and $\Omega$). In order to prove \eqref{odd-log} we will use the exponential inequality  \eqref{mosertrudinger} which describes the limiting behavior of the Sobolev inequality \eqref{sobolev0} when $N=2m-2k-2$ for functions in $H^{m-k-1}(\Omega)$. Let then $f\in H^{m-k-1}(\Omega)$ and let $C_1,C_2$ be the constants appearing in \eqref{mosertrudinger}. We have
\begin{multline*}
\int_{B(0,\varepsilon)} f(x)^2dx=\frac{|B(0,\varepsilon)|}{C_1}\|f\|_{H^{m-k-1}(\Omega)}^2\int_{B(0,\varepsilon)}C_1\left(\frac{f(x)}{\|f\|_{H^{m-k-1}(\Omega)}}\right)^2\frac{dx}{|B(0,\varepsilon)|}\\
= \frac{|B(0,\varepsilon)|}{C_1}\|f\|_{H^{m-k-1}(\Omega)}^2\int_{B(0,\varepsilon)}\log\left(e^{C_1\left(\frac{f(x)}{\|f\|_{H^{m-k-1}(\Omega)}}\right)^2}\right)\frac{dx}{|B(0,\varepsilon)|}\\
\leq \frac{|B(0,\varepsilon)|}{C_1}\|f\|_{H^{m-k-1}(\Omega)}^2\log\left(\int_{B(0,\varepsilon)}e^{C_1\left(\frac{f(x)}{\|f\|_{H^{m-k-1}(\Omega)}}\right)^2}\frac{dx}{|B(0,\varepsilon)|}\right)\\
\leq \frac{|B(0,\varepsilon)|}{C_1}\|f\|_{H^{m-k-1}(\Omega)}^2\log\left(\frac{1}{|B(0,\varepsilon)|}\int_{\Omega}e^{C_1\left(\frac{f(x)}{\|f\|_{H^{m-k-1}(\Omega)}}\right)^2}dx\right)\\
\leq \frac{|B(0,\varepsilon)|}{C_1}\|f\|_{H^{m-k-1}(\Omega)}^2\log\left(\frac{C_2}{|B(0,\varepsilon)|}\right)\\
=\frac{|B(0,\varepsilon)|}{C_1}\|f\|_{H^{m-k-1}(\Omega)}^2\left(\log(C_2)-\log(|B(0,\varepsilon)|)\right)\\
\leq C\varepsilon^{2m-2k-2}(1+|\log(\varepsilon)|) \|f\|_{H^{m-k-1}(\Omega)}^2,
\end{multline*}
where in the first inequality we have used the concavity of the logarithm and Jensen's inequality and in the third inequality we have applied \eqref{mosertrudinger}. Inequality \eqref{odd-log} is now proved.

Let now $u\in C^{k+1}(\Omega)\cap H^m(\Omega)$. From the Sobolev inequality \eqref{sobolev-C}, it follows that $u\in C^{k+1}(\Omega)\cap C^{k,\gamma}(\Omega)$ for all $\gamma\in]0,1[$.
From Taylor's Theorem (see also \eqref{taylor0}) it follows it follows then
\begin{multline}\label{odd-0}
\int_{B(0,\varepsilon)}\left|u(x)-\sum_{|\alpha|\leq k}\frac{\partial^{\alpha} u(0)}{\alpha !}x^{\alpha}\right|^2dx\\
\leq\sum_{|\beta|=k+1}\frac{(k+1)!}{(\beta !)^2}\int_{B(0,\varepsilon)}\int_0^1(1-t)^{2k}|\partial^{\beta}u(tx)|^2|x|^{2(k+1)}dtdx. 
\end{multline}
We estimate the integrals appearing in the right-hand side of \eqref{odd-0}
\begin{multline}\label{odd-1}
\int_{B(0,\varepsilon)}\int_0^1(1-t)^{2k}|\partial^{\beta}u(tx)|^2|x|^{2(k+1)}dtdx\\
=\int_0^1(1-t)^{2k}\int_{B(0,\varepsilon)}|\partial^{\beta}u(tx)|^2|x|^{2(k+1)}dxdt\\
=\int_0^1t^{-2m}(1-t)^{2k}\int_{B(0,t\varepsilon)}|\partial^{\beta}u(y)|^2|y|^{2(k+1)}dxdt\\
\leq \int_0^1t^{-2m+2k+2}\varepsilon^{2k+2}(1-t)^{2k}\|\partial^{\beta}u\|_{L^2(B(0,t\varepsilon))}^2dt\\
\leq C\varepsilon^{2m}(1+|\log(\varepsilon)|)\|u\|_{H^m(\Omega)}^2,
\end{multline}
where the last inequality follows from \eqref{odd-log} applied with $f=\partial^{\beta}u$. From \eqref{odd-0} and \eqref{odd-1}, $ii)$ immediately follows. This ends the proof of the lemma. 
\endproof
\end{lemma}

Assume now that $u\in H^m(\Omega)$ is such that $\int_{\Omega}\tilde\rho_{\varepsilon}u(x)x^{\alpha}dx=0$ for all $|\alpha|\leq m-1$, that is, for a fixed $\delta\in]0,1/2[$
$$
\int_{\Omega}u(x)x^{\alpha}dx+\varepsilon^{-2m+\delta}\int_{B(0,\varepsilon)}u(x)x^{\alpha}dx=0
$$
for all $|\alpha|\leq m-1$. In view of Lemma \ref{pre2} we expect that the quantities $|\partial^{\alpha}u(0)|$, $\left(\int_{B(0,\varepsilon)}u^2dx\right)^{\frac{1}{2}}$ and $\int_{\Omega}u(x)x^{\alpha}dx$ can be bounded by the $\|u\|_{H^m(\Omega)}$ and a suitable power of $\varepsilon$. In particular we expect that all these quantities vanish as $\varepsilon\rightarrow 0^+$. The aim of the next lemma is to prove that this is exactly what happens. We shall also provide the correct powers of $\varepsilon$ in the estimates which are crucial in the proof of Theorem \ref{counter23}.

\begin{lemma}\label{meanlemma}
Let $\Omega$ be a bounded domain in $\mathbb R^N$, $N<2m$, with Lipschitz boundary. Assume that $0\in\Omega$ and let $\varepsilon>0$ be such that $B(0,\varepsilon)\subset\subset\Omega$. Let $\delta\in]0,1/2[$ be fixed. Let $u\in H^m(\Omega)$ be such that $\int_{\Omega}\tilde\rho_{\varepsilon} u(x) x^{\beta} dx=0$ for all $\beta\in\mathbb N^N$ with $|\beta|\leq m-1$. Then there exists a positive constant $C$ which depends only on $m,k$ and $\Omega$ such that
\begin{enumerate}[i)]
\item $|\partial^{\alpha}u(0)|\leq C\varepsilon^{k+\frac{1}{2}-|\alpha|}\|u\|_{H^m(\Omega)}$ for all $\alpha\in\mathbb N^N$ with $|\alpha|\leq k$, if $N=2m-2k-1$, $\forall\,k=0,...,m-1$;
\item $|\partial^{\alpha}u(0)|\leq C\varepsilon^{k+1-|\alpha|}(1+|\log(\varepsilon)|)^{\frac{1}{2}}\|u\|_{H^m(\Omega)}$ for all $\alpha\in\mathbb N^N$ with $|\alpha|\leq k$,  if $N=2m-2k-2$, $\forall\,k=0,...,m-2$;
\item $\int_{B(0,\varepsilon)}u^2dx\leq C\varepsilon^{2m}\|u\|_{H^m(\Omega)}^2$ if $N=2m-2k-1$, $\forall\,k=1,...,m-1$;
\item $\int_{B(0,\varepsilon)}u^2dx\leq C\varepsilon^{2m} (1+|\log(\varepsilon)|)\|u\|_{H^m(\Omega)}^2$ if $N=2m-2k-2$, $\forall\,k=1,...,m-2$;
\item $\left|\int_{\Omega}u x^{\alpha}dx\right|\leq C\varepsilon^{|\alpha|+\delta-k-\frac{1}{2}}\|u\|_{H^m(\Omega)}$ for all $\alpha\in\mathbb N^N$ with $k+1\leq|\alpha|\leq m-1$,  if $N=2m-2k-1$, $\forall\,k=0,...,m-2$;
\item $\left|\int_{\Omega}u x^{\alpha}dx\right|\leq C\varepsilon^{|\alpha|+\delta-k-1}(1+|\log(\varepsilon)|)^{\frac{1}{2}}\|u\|_{H^m(\Omega)}$ for all $\alpha\in\mathbb N^N$ with $k+1\leq|\alpha|\leq m-1$,  if $N=2m-2k-2$, $\forall\,k=0,...,m-2$.
\end{enumerate}
\proof
Let $u\in H^m(\Omega)$ be such that $\int_{\Omega}\tilde\rho_{\varepsilon} u(x) x^{\beta} dx=0$ for all $\beta\in\mathbb N^N$ with $|\beta|\leq m-1$. We start by proving $i)$ and $ii)$. We recall that by Sobolev inequality \eqref{sobolev-C}, $u\in C^{k,\frac{1}{2}}(\Omega)$ if $N=2m-2k-1$, while $u\in C^{k,\gamma}(\Omega)$ for all $\gamma\in]0,1[$ if $N=2m-2k-2$.  We have, for $|\beta|\leq m-1$, that
\begin{multline}\label{ultimo}
\int_{B(0,\varepsilon)}\tilde\rho_{\varepsilon}u(x) x^{\beta}\\
=\int_{B(0,\varepsilon)}\tilde\rho_{\varepsilon}\left(u(x)-\sum_{|\alpha|\leq k}\frac{\partial^{\alpha}u(0)}{\alpha !}x^{\alpha}\right) x^{\beta}dx+\sum_{|\alpha|\leq k}\frac{\partial^{\alpha}u(0)}{\alpha !}\int_{B(0,\varepsilon)}\tilde\rho_{\varepsilon}x^{\alpha} x^{\beta}dx.
\end{multline}
From \eqref{ultimo} and from the fact that $\int_{\Omega}\tilde\rho_{\varepsilon} u(x) x^{\beta} dx=0$ it follows that
\begin{multline}\label{0mean} 
\sum_{|\alpha|\leq k}\frac{\partial^{\alpha}u(0)}{\alpha !}\varepsilon^{-2m+\delta}\int_{B(0,\varepsilon)}x^{\alpha} x^{\beta}dx\\
=-\int_{\Omega}u(x)x^{\beta}dx-\varepsilon^{-2m+\delta}\int_{B(0,\varepsilon)}\left(u(x)-\sum_{|\alpha|\leq k}\frac{\partial^{\alpha}u(0)}{\alpha !}x^{\alpha}\right) x^{\beta}dx.
\end{multline}
We compute now $\int_{B(0,\varepsilon)}\varepsilon^{-2m+\delta}x^{\alpha}x^{\beta}dx$. It is convenient to pass to the spherical coordinates $(r,\theta)=(r,\theta_1,...\theta_{N-1})\in[0,+\infty[\times\partial B$, where $\partial B$ denotes the unit sphere in $\mathbb R^N$ endowed with the $N-1$ dimensional volume element $d\sigma(\theta)$. With respect to these  new variables, we write the coordinate functions $x_i$ as $x_i=r H_i(\theta)$, where $H_i(\theta)$ are the standard spherical harmonics of degree $1$ in $\mathbb R^N$. We have then
\begin{multline*}
\int_{B(0,\varepsilon)}\varepsilon^{-2m+\delta}x^{\alpha}x^{\beta}dx\\
=\int_{\partial B}\int_0^{\varepsilon}H_1^{\alpha_1+\beta_1}\cdots H_N^{\alpha_N+\beta_N}r^{|\alpha|+|\beta|}r^{N-1}drd\sigma(\theta)=K_{\alpha,\beta}\varepsilon^{N+|\alpha|+|\beta|-2m+\delta},
\end{multline*}
where
\begin{equation}\label{Kab}
K_{\alpha,\beta}=\frac{1}{N+|\alpha|+|\beta|}\int_{\partial B}H_1^{\alpha_1+\beta_1}\cdots H_N^{\alpha_N+\beta_N}d\sigma(\theta).
\end{equation}
From \eqref{0mean} it follows that for all $|\beta|\leq k$
\begin{multline}\label{der-system}
\sum_{|\alpha|\leq k}\frac{K_{\alpha,\beta}}{\alpha !}\varepsilon^{|\alpha|}\partial^{\alpha}u(0)\\
=-\varepsilon^{2m-N-|\beta|-\delta}\int_{\Omega}u(x)x^{\beta}dx-\varepsilon^{-N-|\beta|}\int_{B(0,\varepsilon)}\left(u(x)-\sum_{|\alpha|\leq k}\frac{\partial^{\alpha}u(0)}{\alpha !}x^{\alpha}\right)x^{\beta}dx.
\end{multline}
Clearly
\begin{equation}\label{der-O}
\left|\int_{\Omega}u(x)x^{\beta}dx\right|\leq C_{\Omega,\beta}\|u\|_{L^2(\Omega)},
\end{equation}
where $C_{\Omega,\beta}$ depends only on $\Omega$ and $|\beta|$. Moreover, from Lemma \ref{pre2} and H\"older's inequality it follows that
\begin{equation}\label{der-even}
\left|\int_{B(0,\varepsilon)}\left(u(x)-\sum_{|\alpha|\leq k}\frac{\partial^{\alpha}u(0)}{\alpha !}x^{\alpha}\right)x^{\beta}dx\right|\leq C\varepsilon^{|\beta|+2m-k-\frac{1}{2}}\|u\|_{H^m(\Omega)},
\end{equation}
if $N=2m-2k-1$, while
\begin{equation}\label{der-odd}
\left|\int_{B(0,\varepsilon)}\left(u(x)-\sum_{|\alpha|\leq k}\frac{\partial^{\alpha}u(0)}{\alpha !}x^{\alpha}\right)x^{\beta}dx\right|\leq C\varepsilon^{|\beta|+2m-k-1}(1+|\log(\varepsilon)|)^{\frac{1}{2}}\|u\|_{H^m(\Omega)},
\end{equation}
if $N=2m-2k-2$.
From \eqref{der-system}, \eqref{der-O}, \eqref{der-even}, \eqref{der-odd} and since $\delta\in]0,1/2[$, we deduce that
\begin{multline*}
\sum_{|\alpha|\leq k}\frac{K_{\alpha,\beta}}{\alpha !}\varepsilon^{|\alpha|}\partial^{\alpha}u(0)\leq C (\varepsilon^{2k+1-|\beta|-\delta}+\varepsilon^{k+\frac{1}{2}})\|u\|_{H^m(\Omega)}\\
\leq 2C \varepsilon^{k+\frac{1}{2}}\|u\|_{H^m(\Omega)},
\end{multline*}
if $N=2m-2k-1$, while
\begin{multline*}
\sum_{|\alpha|\leq k}\frac{K_{\alpha,\beta}}{\alpha !}\varepsilon^{|\alpha|}\partial^{\alpha}u(0)\leq C (\varepsilon^{2k+2-|\beta|-\delta}+\varepsilon^{k+1}(1+|\log(\varepsilon)|^{\frac{1}{2}}))\|u\|_{H^m(\Omega)}\\
\leq 2C \varepsilon^{k+1}(1+|\log(\varepsilon)|^{\frac{1}{2}})\|u\|_{H^m(\Omega)},
\end{multline*}
if $N=2m-2k-2$. Since for all $|\alpha|\leq k$, $K_{\alpha,\alpha}>0$ by definition (see \eqref{Kab}), necessairily inequalities $i)$ and $ii)$ must hold with a constant $C>0$ which depends only on $m,k$ and $\Omega$.

Consider now $iii)$. We have
\begin{multline*}
\int_{B(0,\varepsilon)}u^2dx=\int_{B(0,\varepsilon)}\left(u(x)-\sum_{|\alpha|\leq k}\frac{\partial^{\alpha}u(0)}{\alpha !}x^{\alpha}+\sum_{|\alpha|\leq k}\frac{\partial^{\alpha}u(0)}{\alpha !}x^{\alpha}\right)^2dx\\
\leq 2 \int_{B(0,\varepsilon)}\left(u(x)-\sum_{|\alpha|\leq k}\frac{\partial^{\alpha}u(0)}{\alpha !}x^{\alpha}\right)^2dx\\
+2\binom{N+k}{k}\sum_{|\alpha|\leq k}\frac{|\partial^{\alpha}u(0)|^2}{(\alpha !)^2}\int_{B(0,\varepsilon)}|x|^{2\alpha}dx\\
\leq C\varepsilon^{2m}\|u\|_{H^m(\Omega)}^2+C\sum_{|\alpha|\leq k}\varepsilon^{2k+1-2\alpha}\int_{B(0,\varepsilon)}|x|^{2\alpha}dx\\
\leq C\varepsilon^{2m}\|u\|_{H^m(\Omega)}^2,
\end{multline*}
where in the second inequality we have used point $i)$ of Lemma \ref{pre2} and in the last inequality we have used point $i)$ of the present lemma and the fact that
$$
\left|\int_{B(0,\varepsilon)}x^{2\alpha}dx\right|\leq\varepsilon^{2|\alpha|}|B(0,\varepsilon)|=\omega_{2m-2k-1}\varepsilon^{2|\alpha|+2m-2k-1}.
$$
This concludes the proof of $iii)$. Point $iv)$ is proved exactly as point $iii)$, by using point $ii)$ of Lemma \ref{pre2} and point $ii)$ of the present lemma.

We consider now point $v)$. Let $N=2m-2k-1$, with $0\leq k\leq m-2$ and let $\beta\in\mathbb N^N$ such that $k+1\leq|\beta|\leq m-1$. We have
\begin{multline}\label{mean0}
\int_{\Omega}u(x)x^{\beta}dx=-\varepsilon^{-2m+\delta}\int_{B(0,\varepsilon)}u(x)x^{\beta}dx\\
=-\varepsilon^{-2m+\delta}\int_{B(0,\varepsilon)}\left(u(x)-\sum_{|\alpha|\leq k}\frac{\partial^{\alpha}u(0)}{\alpha !}x^{\alpha}\right)x^{\beta}dx\\
-\varepsilon^{-2m+\delta}\sum_{|\alpha|\leq k}\frac{\partial^{\alpha}u(0)}{\alpha !}\int_{B(0,\varepsilon)}x^{\alpha}x^{\beta}dx.
\end{multline}
From Lemma \ref{pre2} point $i)$ and from H\"older's inequality, we have that
\begin{multline}\label{mean1}
\varepsilon^{-2m+\delta}\int_{B(0,\varepsilon)}\left|\left(u(x)-\sum_{|\alpha|\leq k}\frac{\partial^{\alpha}u(0)}{\alpha !}x^{\alpha}\right)x^{\beta}\right|dx\\
\leq \varepsilon^{-2m+\delta}\left(\int_{B(0,\varepsilon)}(u(x)-\sum_{|\alpha|\leq k}\frac{\partial^{\alpha}u(0)}{\alpha !}x^{\alpha})^2dx\right)^{\frac{1}{2}}\left(\int_{B(0,\varepsilon)}|x|^{2|\beta|}dx\right)^{\frac{1}{2}}\\
\leq C\varepsilon^{|\beta|+\delta-k-\frac{1}{2}}\|u\|_{H^m(\Omega)}.
\end{multline}
Moreover, from point $i)$ of the present lemma we have that  for all $|\alpha|\leq k$
\begin{equation}\label{mean2}
\varepsilon^{-2m+\delta}\frac{|\partial^{\alpha}u(0)|}{\alpha !}\int_{B(0,\varepsilon)}|x|^{|\alpha|+|\beta|}dx\leq C\varepsilon^{|\beta|+\delta-k-\frac{1}{2}}.
\end{equation}
The proof of $v)$ follows by combining \eqref{mean0} with \eqref{mean1} and \eqref{mean2}. The proof of $vi)$ is identical to that of $v)$ and follows form point $ii)$ of Lemma \ref{pre2} and point $ii)$ of the present lemma. This concludes the proof.

\endproof
\end{lemma}

\bibliography{bibliography}{}
\bibliographystyle{abbrv}

\end{document}